\pdfoutput=1 
\documentclass[preprint,12pt]{elsarticle}
\usepackage[margin=25mm]{geometry}
\usepackage{amssymb}
\usepackage{amsmath}
\usepackage{amsmath,amsfonts,amssymb}
\usepackage{amsthm}
\usepackage[noend]{algpseudocode}
\usepackage{algorithmicx,algorithm}
\usepackage{graphicx,subfigure}
\usepackage{hyperref}

\numberwithin{equation}{section}

\newtheorem{assumption}{Assumption}[section]
\newtheorem{lemma}[assumption]{Lemma}
\newtheorem{theorem}[assumption]{Theorem}
\newtheorem{proposition}[assumption]{Proposition}

\newtheorem{corollary}[assumption]{Corollary}
\newtheorem{definition}[assumption]{Definition}

\journal{Arxiv}

\begin{document}

\begin{frontmatter}



\title{Policy Learning for Perturbance-wise Linear Quadratic Control Problem} %% Article title


%\author{} %% Author name
%
%%% Author affiliation
%\affiliation{organization={},%Department and Organization
            
\author[label1]{Haoran Zhang}
\affiliation[label1]{organization={Department of Applied Mathematics, The Hong Kong Polytechnic University},
             country={Hong Kong}}
             
\author[label2]{Wenhao Zhang}
\affiliation[label2]{organization={Department of Applied Mathematics, The Hong Kong Polytechnic University},
             country={Hong Kong}}
 
\author[label3]{Xianping Wu}            
\affiliation[label3]{organization={School of School of Mathematics and Statistics, Guangdong University of Technology},
             %addressline={},
             city={GuangZhou},
             postcode={510520},
             state={Guangdong},
             country={China}}

%% Abstract
\begin{abstract}
We study finite horizon linear–quadratic control with additive noise in a perturbance-wise framework that unifies the classical model, a constraint embedded affine policy class, and a distributionally robust formulation with a Wasserstein ambiguity set. Based on an augmented affine representation, we model feasibility as an affine perturbation and unknown noise as distributional perturbation from samples, thereby addressing constrained implementation and model uncertainty in a single scheme. First, we construct an implementable policy gradient method that accommodates nonzero noise means estimated from data. Second, we analyze its convergence under constant stepsizes chosen as simple polynomials of problem parameters, ensuring global decrease of the value function. Finally, numerical studies: mean–variance portfolio allocation and dynamic benchmark tracking on real data, validating stable convergence and illuminating sensitivity trade-offs across horizon length, trading cost intensity, state penalty scale, and estimation window.
\end{abstract}



%% Keywords

\end{frontmatter}


\section{Introduction}
Linear–quadratic (LQ) optimal control is a fundamental component of stochastic control, where linear dynamics combined with quadratic performance indices produce feedback laws defined by algebraic riccati equations (ARE) for both finite and infinite horizons. Fundamental research delineated the structural characteristics of optimum regulators and the significance of the ARE in stochastic contexts, offering a robust framework for study and applications in engineering and finance \cite{kalman1960contributions,wonham1968matrix}.

Realistic deployments impose actuator, budget, or regulatory limits. A major line of work shows that constrained LQ problems admit explicit offline solutions in the form of piecewise-affine feedback laws and explicit controller characterizations in parametric explicit Model Predictive Control (MPC), which utilize through a scalar constrained example to demonstrate structure \cite{bemporad2002explicit,wu2018explicit}. Other advances refine computation and implementation, such as binary search trees, lattice/regular partitions, hardware synthesis, etc. \cite{tondel2003evaluation,wen2009analytical,kvasnica2011stabilizing}. We will leverage this perspective through a scalar constrained example that decomposes into piecewise classical LQ subproblems. 

Beyond hard constraints, a second pressure is distributional uncertainty: the noise may be unknown, light-tailed, or shift across regimes. Distributionally robust optimization (DRO) addresses such misspecification by optimizing against the worst case distribution in an ambiguity set. Wasserstein ball ambiguity has emerged as a statistically grounded and tractable choice with finite sample guarantees and convex reformulations \cite{mohajerin2018data}. Recent results bring Wasserstein DRO directly into LQ control, proving that an observation linear (and in the zero-mean case, linear) policy remains optimal and providing efficient algorithms via least  favorable distributions \cite{taskesen2023distributionally}. Related work also investigates regret optimal and infinite horizon variants under Wasserstein ambiguity \cite{kargin2024wasserstein}. These developments motivate treating unknown noise via a Wasserstein  penalized and minimax formulation in LQ related studies.

Instead of solving Riccati equations model-based, reinforcement learning (RL) frames control design as policy optimization: directly update a parameterized feedback to minimize the expected finite horizon LQ cost via gradient steps on the performance objective. A key insight is that, despite nonconvexity, the LQ landscape is benign, for time invariant LQ, policy gradient achieves global convergence to the optimal stabilizing controller under suitable stepsizes and regularity, revealing a clear link between RL’s policy search and classical optimal control \cite{fazel2018global}. These guarantees extend to finite horizon, stochastic settings, which is relevant to our problem, where policy gradient on time-varying gains converges and projected variants accommodate constraints \cite{hambly2021policy}. When gradients are estimated from rollouts, derivative-free policy gradient achieves polynomial sample complexity over linear policies tailored to LQ systems, further cementing LQ as a precise proving ground for RL methods \cite{malik2020derivative}; complementary lines quantify data requirements for learning LQ controllers from limited experiments \cite{dean2020sample}. \\

\noindent \textbf{Contributions}. We study policy gradient methods for finite horizon LQ control with additive noise. The analysis covers a scalar  state constrained case, revealing a piecewise-LQ structure and a DRO extension in which the unknown noise distribution is modeled via a Wasserstein ambiguity set calibrated from samples. We introduce an augmented affine representation that turns feasibility constraints into policy structure and accommodates nonzero noise means, and we show that the resulting policy gradient admits an exact, implementable form. We establish a global convergence guarantee for constant stepsizes chosen as low  order polynomials of problem parameters, and we connect Riccati optimality with gradient based updates under both constrained and DRO formulations. Empirically, we compare constrained and unconstrained solutions in a mean–variance portfolio setting and evaluate the DRO policy on a dynamic benchmark, tracking task with sensitivity analyses over horizon, trading cost intensity, state penalty scale, and estimation window. \\

\noindent \textbf{Notations and Organization}. Vectors and matrices are real. $\mathbb{R}^n$ and $\mathbb{R}^{n\times m}$ denote Euclidean spaces; $\|\cdot\|_2$ and $\|\cdot\|_F$ are spectral and Frobenius norms; for a matrix $Z$, $\sigma_{\min}(Z)$ is its smallest singular value and $\operatorname{Tr}(\cdot)$ is the trace; $\mathbf{I}_n$ and $\mathbf{0}_{n\times m}$ are identity and zero matrices, especially, $\mathbf{1}$ is $\mathbf{I}_n$ with $n=1$; $Z^\top$ denotes transpose. For a sequence $\mathbf{D}=(D_0,\ldots,D_T)$, write $|||\mathbf{D}|||=\sum_{t=0}^T\|D_t\|$. For a symmetric matrix $A$, we write $A\succ 0$ (resp.\ $A\succeq 0$) to denote positive definite (resp.\ semidefinite).

The rest of the paper proceeds as follows. Section 2 formulates the finite horizon LQ problem with additive noise, presents the scalar constrained example, and introduces the   Wasserstein penalized DRO model. Section 3 develops the augmented-affine policy gradient and proves global convergence under a constant stepsize. Section 4 reports numerical results: a mean–variance study contrasting constrained and unconstrained cases, and a real data benchmark tracking application that evaluates the DRO controller and conducts targeted sensitivity analyses. Section 5 concludes and outlines extensions.

\section{Problem Setup} \label{ch3-setup}
\subsection{Classical LQ Problem}
Let the set of admissible policies be
$$
\Pi=\{\pi: \mathcal{X} \rightarrow \mathcal{P}(\mathcal{U})\},
$$
where $\mathcal{X}$ denotes the state space, $\mathcal{U}$ the action space, and $\mathcal{P}(\mathcal{U})$ the set of probability measures on $\mathcal{U}$. Each policy $\pi \in \Pi$ assigns to every state $x \in \mathcal{X}$ a probability distribution over actions in $\mathcal{U}$.

The decision maker seeks an optimal policy by solving
\begin{equation}
\min _{\pi \in \Pi} \mathbb{E}_{x \sim \mathcal{D}}\left[J_\pi(x)\right],
\label{eqp-2.1.1}
\end{equation}
where the value function $J_{\pi}$ in the finite-horizon setting is given by
\begin{equation} \label{mainfunc}
J_\pi(x) =\mathbb{E}_\pi \left[\sum _{t=0}^{T-1}\left( x_{t}^{\top } Q_{t}  x_{t}  +u_{t}^{\top } R_{t} u_{t}\right) +x_{T}^{\top } Q_{T}  x_{T}\Bigg| x_0=x  \right],\end{equation}
for $t = 0,1,\cdots, T-1$, and the state dynamics subject to
\begin{equation}
\begin{aligned}
x_{t+1} = A_{t}x_{t} + B_{t}u_{t} + D_t w_{t}, \ x_0 \sim \mathcal{D}.
\label{Xt+1}
\end{aligned}
\end{equation}
Consider a discrete-time system with state vector $x_t \in \mathcal{X} =\mathbb{R}^m$ and control input $u_t \in \mathcal{U}= \mathbb{R}^n$. The initial state $x_0$ is drawn from a probability distribution $\mathcal{D}$ on $\mathbb{R}^m$. The matrices $A \in \mathbb{R}^{m \times m}$ and $B \in \mathbb{R}^{m \times n}$ specify the linear dynamics of the system, while $Q \in \mathbb{R}^{m \times m}$ and $R \in \mathbb{R}^{n \times n}$ are positive definite weighting matrices appearing in the quadratic cost. The process is driven by a noise sequence $w_t \in \mathbb{R}^k, t=0, \ldots, T-1$, where $\{w_t\}_{t=0}^{T-1}$ are i.i.d., have zero-mean, and are independent of $x_0$. Moreover, the noise has finite second moment, that is, $\mathbb{E}\left[w_t w_t^{\top}\right]=W$, with$\operatorname{Tr}(W)<\infty$ for $t=0,1, \ldots, T-1$.

We now apply the dynamic programming algorithm to define a sequence of positive numbers $\{P_t^*\}_{t=0}^T$ as the solution to the dynamic ARE: %we have 
\begin{equation}
%\begin{aligned}
P_{t}^{*} = A_{t}^{\top}P_{t+1}^{*}A_{t} - A_{t}^{\top}P_{t+1}^{*} B_{t}(B_{t}^{\top } P_{t+1}^{*} B_{t} + R_{t} )^{-1}B_{t}^{\top}P_{t+1}^{*}A_{t}  + Q_{t},
%\end{aligned}
\label{pstar}
\end{equation}
with terminal condition $P_T^{*}=Q_T$. The number $\left\{P_t^{*}\right\}_{t=0}^{T}$ can be found by solving the ARE iteratively backwards in time. \\ 

\noindent The optimal control sequence $\{u_t^*\}_{t=0}^{T-1}$ is given by
\begin{align} \label{5}
u_t^{*} &= -K_t^{*}t  \ \ \ \ \ where \\ 
K_t^* &= (B_{t}^{\top } P_{t+1}^{*} B_{t} + R_{t} )^{-1}
B_{t}^{\top}P_{t+1}^{*}A_{t}. \label{6}
\end{align}
To identify the optimal solution in the linear feedback form (\ref{5}), it is essential to concentrate on the specific class of linear \textit{admissible policies} in feedback form
\begin{equation}
u_t=-K_t x_t, \quad t=0,1, \cdots, T-1.
\end{equation}

\subsection{LQ Problem with Constraints}
When $x$ is a scalar state, we can take into account the following set of general control constraints then find the optimal solution for this case, which is a piece-wise linear function, referred from \cite{wu2018explicit}:
\begin{equation}
\mathcal{C}_t\left(x_t\right)=\left\{u_t \in \mathbb{R}^n\left|H_t u_t \leq d_t\right| x_t \right| \},
\label{constr}
\end{equation}
where $H_t \in \mathbb{R}^{s\times n}$ and $d_t \in \mathbb{R}^s$ are given matrices and vectors. In order to ensure the feasibility, we impose the following standing assumption: for each $t=0, \cdots T-1$, the set $\left\{K \in \mathbb{R}^{n} \mid H_t K \leq d_t\right.$, $\left. H_t K \leq \mathbf{0}_{s \times 1}\right\}$ is nonempty.  The parametrization in \eqref{constr} is flexible and can capture a variety of practically relevant constraints. A few representative examples are:
\begin{itemize}
\item[(i)] Non-negativity/non-positivity. If we choose $H_t=-\mathbf{I}_n$ and $d_t=\mathbf{0}_{n \times 1}$, then (\ref{constr}) enforces $u_t \geq \mathbf{0}_{n \times 1}$(componentwise). Similarly, taking $H_t=\mathbf{I}_n$ and $d_t=\mathbf{0}_{n \times 1}$ yields $u_t \leq \mathbf{0}_{n \times 1}$.
\item[(ii)] State-dependent box constraints. Suppose we are given vectors $\underline{d}_t, \overline{d}_t \in \mathbb{R}^{n}$. The requirement $\underline{d}_t\left|x_t\right| \leq u_t \leq \overline{d}_t\left|x_t\right|$ can be represented by setting $H_t=\begin{bmatrix}\mathbf{I}_n \\ -\mathbf{I}_n\end{bmatrix}$ and $d_t=\begin{bmatrix}\overline{d}_t\\ -\underline{d}_t\end{bmatrix}$.
\item[(iii)] Cone-type constraints. A general conic constraint of the form $H_t u_t \geq \textbf{0}_{s \times 1}$, for a prescribed matrix $H_t$, is also encompassed by this framework..
\item[(iv)] Unconstrained controls. The unconstrained case $u_t \in \mathbb{R}^n$, is recovered by taking $H_t=\mathbf{0}_{s \times n}$ and $d_t=\mathbf{0}_{s \times 1}$, so that \eqref{constr} imposes no restriction on $u_t$.
\end{itemize}
Before deriving the explicit solution to the problem, it is convenient to introduce the following state-independent set associated with $\mathcal{C}_t\left(x_t\right)$: $\mathcal{K}_t=\left\{K \in \mathbb{R}^n \mid H_t K \leq d_t\right\}$. Additionally, three auxiliary optimization problems are introduced for the interval as follows,
\begin{equation}
\begin{aligned}
&\min _{u \in \mathcal{C}_t(\xi)} p_t(u, \xi, y, z),\\
&\min _{K \in \mathcal{K}_t} \hat{p}_t(K, y, z),\\
&\min _{K \in \mathcal{K}_t} \bar{p}_t(K, y, z),
\end{aligned}
\end{equation}
where $\xi \in \mathbb{R}$ is $\mathcal{F}_t$-measurable random variable, and $y, z \in \mathbb{R}_{+}$ are $\mathcal{F}_{t+1}$-measurable random variables. The corresponding objective functions are defined by
\begin{equation}
\begin{aligned}
p_t(u, \xi, y, z)&=  \mathbb{E}_t \Bigl[ \xi^{2 } Q_{t} +u^{\top } R_{t} u+\left(A_{t}\xi + B_{t}u + D_{t}w_t\right)^{2}  \\
& \quad  \times\left(y \mathbf{1}_{\left\{A_{t}\xi + B_{t}u + D_{t}w_t \geq 0\right\}} + z \mathbf{1}_{\left\{A_{t}\xi + B_{t}u + D_{t}w_t< 0 \right\}}\right) \Bigr],
\end{aligned}
\end{equation}
\begin{equation}
\begin{aligned}
\hat{p}_t(K, y, z) &=  \mathbb{E}_t\Bigl[Q_{t} +K^{\top } R_{t} K + \left(A_{t} + B_{t}K +D_{t}w_t\right)^{2} \\
& \quad \times\left(y \mathbf{1}_{\left\{A_{t} + B_{t}K + D_{t}w_t \geq 0\right\}}+z \mathbf{1}_{\left\{A_{t} + B_{t}K + D_{t}w_t< 0 \right\}}\right)\Bigr], 
\end{aligned}
\label{hatp}
\end{equation}
\begin{equation}
\begin{aligned}
\bar{p}_t(K, y, z)&=  \mathbb{E}_t\Bigl[Q_{t} +K^{\top } R_{t} K + \left(A_{t} - B_{t}K - D_{t}w_t\right)^{2} \\
& \quad \times\left(y \mathbf{1}_{\left\{A_{t} - B_{t}K - D_{t}w_t \leq 0\right\}}+z \mathbf{1}_{\left\{A_{t} - B_{t}K - D_{t}w_t> 0 \right\}}\right)\Bigr]. \\
\end{aligned}
\label{barp}
\end{equation}
An explicit solution for the problem can be developed. We introduce two random sequences, $\hat{P}_0, \hat{P}_1, \cdots, \hat{P}_T$ and $\bar{P}_0, \bar{P}_1, \cdots, \bar{P}_T$, defined recursively in reverse as follows:
\begin{equation}
\hat{P}_t=\min _{K_t \in \mathcal{K}_t} \hat{p}_t\left(K_t, \hat{P}_{t+1}, \bar{P}_{t+1}\right),
\label{hatP}
\end{equation}
\begin{equation}
\bar{P}_t=\min _{K_t \in \mathcal{K}_t} \bar{p}_t\left(K_t, \hat{P}_{t+1}, \bar{P}_{t+1}\right),
\label{barP}
\end{equation}
where $\hat{p}_t(\cdots)$ and $\bar{p}_t(\cdots)$ are defined, respectively, in (\ref{hatp}) and (\ref{barp}) for $t=T-1, \cdots, 0$ with the boundary conditions of $\hat{P}_T$ $=\bar{P}_T=Q_T$. It is evident that $\hat{P}_t$ and $\bar{P}_t$ are $\mathcal{F}_t$-measurable random variables. The optimal control policy for the problem at time $t$ is characterized by a linear feedback policy,
\begin{equation}
u_t^*\left(x_t\right)=\left\{\begin{array}{ll}
\hat{K}_t^* x_t & \mbox { if } x_t \geq 0, \\
-\bar{K}_t^* x_t & \mbox { if } x_t<0
\end{array},\right.
\label{real-k-star}
\end{equation}
where $\hat{K}_t^*$ and $\bar{K}_t^*$ are defined as
$$
\begin{aligned}
& \hat{K}_t^*=\arg \min _{K_t \in \mathcal{K}_t} \hat{p}_t\left(K_t, \hat{P}_{t+1}, \bar{P}_{t+1}\right), \\
& \bar{K}_t^*=\arg \min _{K_t \in \mathcal{K}_t} \bar{p}_t\left(K_t, \hat{P}_{t+1}, \bar{P}_{t+1}\right),
\end{aligned}
$$
where $\{\hat{P}_t\}|_{t=0} ^T$ and $\{\bar{P}_t\}|_{t=0} ^T$ are given in (\ref{hatP}) and (\ref{barP}), respectively, and $\hat{P}_t>0$ and $\bar{P}_t>0$ for $t=0, \cdots, T-1$. Furthermore, the optimal objective value of the problem is
\begin{equation}
x_0^2\left(\hat{P}_0 \mathbf{1}_{\left\{x_0 \geq 0\right\}}+\bar{P}_0 \mathbf{1}_{\left\{x_0< 0 \right\}}\right).
\end{equation}
The above theory suggests that the systems represented by (\ref{hatP}) and (\ref{barP}) serve a function analogous to that of the ARE in the classical LQ optimal problem. Moreover, in the absence of control constraints, these two equations will converge into the single equation presented in (\ref{pstar}) later on. 

\subsection{LQ Problem with Unknown Noise Distribution} 
Consider the stochastic program
\begin{equation}
J^* = \min _{\pi \in \Pi} \mathbb{E}_{\mathbb{P}}\left[J_\pi\left(x,u| w\right)\right],
\label{noise_main}
\end{equation}
where $\mathbb{P}$ is the noise distribution supported on $\Omega$.In applications, the distribution $\mathbb{P}$ cannot be accessed explicitly. Instead Instead, we assume that we observe a finite number $M$ i.i.d.\ samples from $\mathbb{P}$ and collect them in the training set $\widehat{\Omega}_M=\{\widehat w_i\}_{i=1}^M\subseteq\Omega$. Before it is realized, this dataset can be regarded as a random element in $\Omega^M$ with distribution $\mathbb{P}^M$.

A data-based control decision $\widehat u_M$ iis any feasible control rule constructed from the information contained in $\widehat{\Omega}_M$. Its “true performance is measured by the expected cost $\mathbb{E}_{\mathbb{P}}\left[J_\pi\left(x, \widehat{u}_M | w\right)\right]$, that is, the expectation with respect to a fresh noise draw $w$, independent of the training set. Since $\mathbb{P}$ is unknown, we aim to derive high-confidence upper bounds on this out-of-sample cost. Concretely, for a random bound $\widehat J_M$ and a prescribed significance level $\beta$, we would like to ensure that
\begin{equation}
\mathbb{P}^M\left\{\widehat{\Omega}_M: \mathbb{E}_{\mathbb{P}}\left[J_\pi\left(x, \widehat{u}_M | w\right)\right] \leq \widehat{J}_M\right\} \geq 1-\beta.
\label{perf-gua}
\end{equation}
Here $\widehat{J}_M$ plays the role of a performance certificate, while the probability on the left-hand side quantifies its reliability.

As it is impossible to evaluate the true out-of-sample performance, we instead seek a data-driven solution with a low certificate and high reliability $\widehat{J}_M$, which are small and reliable in the above sense. A natural first step is to approximate $\mathbb{P}$ by the empirical measure with sample data $\widehat{w}_{t,1}, \cdots, \widehat{w}_{t,M}$ from $\mathbb{P}$:
\begin{equation}
\widehat{\mathbb{P}}_M=\frac{1}{M} \sum_{i=1}^M \delta_{\widehat{w}_{t.i}},
\label{w_sample}
\end{equation}
where $\delta_{\widehat w_i}$ is the Dirac mass at $\widehat w_i$. If one were to treat $\widehat{\mathbb{P}}_M$.as the true distribution in the controller design, the resulting policy could be highly sensitive to sampling noise and may perform poorly when $\mathbb{P}$ differs from $\widehat{\mathbb{P}}_M$. To hedge such misspecification, , we adopt the distributionally robust counterpart
\begin{equation}
\widehat{J}_M =  \min _{\pi \in \Pi} \max_{\mathbb{Q} \in \widehat{\mathcal{P}}_M}  \mathbb{E}_{\mathbb{Q}}\left[ J_\pi\left(x, u | w\right)\right]
\label{DRO}
\end{equation}
where $\widehat{\mathcal{P}}_M$ denotes an ambiguity set of probability measures that are deemed plausible given the observed samples. In this paper, we construct $\widehat{\mathcal{P}}_M$ as a ball centered at the empirical distribution $\widehat{\mathbb{P}}_M$, measured in the Wasserstein metric. We will show that the optimal value $\widehat{\mathcal{P}}_M$ and any optimal data-driven control $\widehat{u}_M$ solving this distributionally robust problem enjoy two desirable properties: a finite-sample performance guarantee (Theorem \ref{finite-sample}) and asymptotic consistency as the sample size $M$ tends to infinity (Theorem \ref{asy-conv}).

\subsubsection{Wasserstein Metric and Ambiguity Set}
Let $\mathcal{M}(\Omega)$ denote the set of all probability measures $\mathbb{Q}$ on $\Omega$ such that $\mathbb{E}_{\mathbb{Q}}[\|w\|] = \int_{\Omega}\|w\| \mathbb{Q}(d w)<\infty$.
\begin{definition}[Wasserstein metric] 
For $\mathbb{Q}_1, \mathbb{Q}_2 \in \mathcal{M}(\Omega)$, the 2-Wasserstein metric $W_2\left(\mathbb{Q}_1, \mathbb{Q}_2\right)$ is defined as
$$
W_2\left(\mathbb{Q}_1, \mathbb{Q}_2\right)= \left(\inf \int_{\Omega \times \Omega}(\left\|w_1-w_2\right\|)^2 \Pi_w \left(d w_1, d w_2\right)\right)^{1/2},
$$
where the infimum is taken over all joint distributions $\Pi_w$ of $w_1$ and $w_2$ whose marginals are $\mathbb{Q}_1$ and $\mathbb{Q}_2$, respectively.
\label{def_w}
\end{definition}
More generally, one can define the $p$-Wasserstein metric ($p\ge 1$) by using the cost $\|w_1-w_2\|^p$. In what follows we restrict attention to $p=2$  To quantify ambiguity set around the empirical law $\widehat{\mathbb{P}}_M$, we work with the $W_2$–neighborhood
$$
\mathbb{B}_{\varepsilon}(\widehat{\mathbb{P}}_M) =\left\{\mathbb{Q} \in \mathcal{M}(\Omega): W_2(\widehat{\mathbb{P}}_M, \mathbb{Q}) \leq \varepsilon_r\right\}.
$$
Given a confidence parameter $\beta\in(0,1)$, we calibrate with the radius $\varepsilon_r$ so that
$$
\mathbb{P}^M\!\left\{\,\mathbb{P}\in \mathbb{B}_{\varepsilon_r}(\widehat{\mathbb{P}}_M)\,\right\}\ \ge\ 1-\beta.
$$
In words, $\mathbb{B}_{\varepsilon_r}(\widehat{\mathbb{P}}_M)$ is a $W_2$–ball centered at the empirical measure which, with probability at least $1-\beta$ over the draw of the $M$ samples, contains the unknown data distribution. In particular, for any confidence $1-\beta>0$, the goal is to choose an ambiguity ball radius $\varepsilon_r$ so that the true distribution in this ball with confidence $1-\beta$. \\

We now impose a mild regularity assumption on the noise distribution.
\begin{assumption}[Light-tailed distribution]
There exists an exponent $a>1$ such that
$$
\mathbb{E}_{\mathbb{P}}\left[\exp \left(\|w\|^a\right)\right]=\int_{\Omega} \exp \left(\|w\|^a\right) \mathbb{P}(d w)<\infty.
$$
\label{light-tail}
\end{assumption}
Assumption \ref{light-tail} essentially requires the tail of the distribution $\mathbb{P}$ to decay at an exponential rate. This ensures that standard measure concentration inequalities apply, allowing us to control the deviation of empirical quantities from their expectations with high probability.

\begin{theorem}[Measure concentration]
If Assumption \ref{light-tail} holds, we have
$$
\mathbb{P}^M\left\{W_2(\mathbb{P}, \widehat{\mathbb{P}}_M) \geq \varepsilon_r\right\} \leq \begin{cases}c_1 \exp \left(-c_2 M \varepsilon_r^{\max \{m, 2\}}\right) & \text { if } \varepsilon_r \leq 1 \\ c_1 \exp \left(-c_2 M \varepsilon_r^a\right) & \text { if } \varepsilon_r>1\end{cases}
$$
for all $M \geq 1$, $m \neq 2$, and $\varepsilon>0$, where $c_1$, $c_2$ are positive constants while $c_2$ depends on the $p$-norm. 
\label{m-con}
\end{theorem}
The proofs of Theorem \ref{m-con} are referred from \cite{fournier2015rate}. 

By Theorem \ref{m-con}, we can estimate the radius of the smallest Wasserstein ball that contains $\mathbb{P}$ with confidence $1-\beta$ for some prescribed $\beta \in(0,1)$ as:  
$$
\varepsilon_M(\beta):= \begin{cases}\left(\frac{\log \left(c_1 \beta^{-1}\right)}{c_2 M}\right)^{1 / \max \{m, 2\}} & \text { if } M \geq \frac{\log \left(c_1 \beta^{-1}\right)}{c_2}, \\ \left(\frac{\log \left(c_1 \beta^{-1}\right)}{c_2 M}\right)^{1 / a} & \text { if } M<\frac{\log \left(c_1 \beta^{-1}\right)}{c_2} .\end{cases}
$$
Note that the Wasserstein ball with radius $\varepsilon_M(\beta)$ can thus be viewed as a confidence set for the unknown true distribution as in statistical testing \cite{mohajerin2018data, wainwright2019high}. We now connect the above construction to the distributionally robust control problem (\ref{DRO}) introduced earlier. 

\begin{theorem}[Finite sample guarantee]
Let $\beta \in(0,1)$ be fixed and define the ambiguity set $\widehat{\mathcal{P}}_M=\mathbb{B}_{\varepsilon_M(\beta)}(\widehat{\mathbb{P}}_M)$. Suppose Assumption \ref{light-tail} holds and let $\widehat{J}_M$ and $\widehat{u}_M$ denote, espectively, the optimal value and an optimal policy of the distributionally robust problem (\ref{DRO}). Then, the finite sample guarantee (\ref{perf-gua}) holds.
\label{finite-sample}
\end{theorem}

We finally discuss the asymptotic behavior as the number of samples grows.
\begin{theorem}[Asymptotic consistency] 
Let $\beta_M$ be a sequence in $(0,1)$ such that $\sum_{M=1}^{\infty} \beta_M<\infty$ and $\lim _{M \rightarrow \infty} \varepsilon_M\left(\beta_M\right)=0$ for $M \in \mathbb{N}$. For each $M$, define the ambiguity set $\widehat{\mathcal{P}}_M=\mathbb{B}_{\varepsilon_M\left(\beta_M\right)}(\widehat{\mathbb{P}}_M)$ and let $\widehat{J}_M$ and $\widehat{u}_M$ represent the optimal value and an optimal policy of the problem (\ref{DRO}). Under Assumption \ref{light-tail}, one can show that $\widehat{J}_M$ goes down to $J^*$ as $M$ goes to $\infty$ where $J^*$ is the optimal value of (\ref{noise_main}).
\label{asy-conv}
\end{theorem}
The proof of Theorem $\ref{finite-sample}$ and Theorem $\ref{asy-conv}$ is based on \cite{mohajerin2018data}. 

\subsubsection{Solve Minimax LQ Problem with Wasserstein Penalty}
When tackling the distributionally robust control problem (\ref{DRO}), it is useful to note that a direct dynamic programming approach over distributions is typically computationally prohibitive in high-dimensions, in contrast to Riccati-type recursions for classical LQ problems. Following the regularization approach in the literature \cite{kim2023distributional}, we instead modify the cost functional by introducing a Wasserstein penalty term.

For a given policy $\pi$, define the penalized cost
$$
J^W_\pi\left(x, u | w\right) = \mathbb{E}_\pi \left[ x_{T}^{\top } Q_{T} x_T + \sum _{t=0}^{T-1}\left( x_{t}^{\top } Q_{t} x_t  +u_{t}^{\top } R_{t} u_{t} - \lambda W_2 (\widehat{\mathbb{P}}_M, \mathbb{Q})^2 \right) \Bigg| x_0=x  \right],
$$
where $\lambda > 0$ is the penalty parameter. By varying $\lambda$, we can tune how conservative the resulting control policy is: larger values of $\lambda$ penalize distributions that are far from the empirical law more heavily. The corresponding penalized minimax problem is
\begin{equation}
\min _{\pi \in \Pi} \max_{\mathbb{Q} \in \widehat{\mathcal{P}}_M}  \mathbb{E}_{\mathbb{Q}}\left[ J_\pi\left(x, u | w\right)\right], 
\label{w-penalty}
\end{equation}
which we use as a tractable surrogate for the original DRO formulation (\ref{DRO}). Let $\pi^*$ denote an optimal policy of (\ref{w-penalty}) for a given $\lambda$, and define
$$
V(x; \lambda)=\inf _{\pi \in \Pi} \sup _{\mathbb{Q} \in \widehat{\mathcal{P}}_M} J^W_\pi\left(x, u | w\right)
$$
the optimal value of the penalized minimax problem. Then, for a suitable choice of $\lambda^*$, one can show a guaranteed-cost property of the form
$$
\sup_{\mathbb{Q} \in \widehat{\mathcal{P}}_M} \left[J^W_\pi\left(x, u | w\right) | \pi^*(\lambda^*) \right] \leq \lambda^* {\varepsilon_M}^2(\beta) + V(x; \lambda^*),
$$
where ${\varepsilon_M}(\beta)$ is the ambiguity radius defined earlier. To analyze this penalized problem over a finite horizon, we introduce the value function
$$
V_t(x) = \inf_{\pi \in \Pi}  \sup _{\mathbb{Q} \in \widehat{\mathcal{P}}_M} \mathbb{E}_{\mathbb{Q}} \left[\sum _{s=t}^{T-1}\left( x_{s}^{\top} Q_{s} x_s +u_{s}^{\top } R_{s} u_{s} - \lambda W_2 (\widehat{\mathbb{P}}_M, \mathbb{Q})^2 \right) + x_{T}^{\top } Q_{T}  x_t \Bigg| x_t=x  \right],
$$
which represents the minimal worst-case expected cost-to-go from stage $t$ onward, given $x_t = x$. By construction, the scalar value $V(x; \lambda)$ can be identified with $V_0(x; \lambda ) / T$. The dynamic programming principle then yields the recursion
$$
V_t(x)=x^{\top} Qx  +\inf_u \sup _{\mathbb{Q} \in \mathcal{M}(\Omega)} \left[u^{\top} R u - \lambda W_2 (\widehat{\mathbb{P}}_M, \mathbb{Q})^2 + \int_{\Omega} V_{t+1}(A_{t}x + B_{t}u + D_{t}w) \mathbb{Q} (dw)\right]
$$
for $t = 0, \cdots, T-1$, and $V_T(x) = x^{2}Q_T$. The inner maximization over probability measures $\mathbb{Q}$ is infinite-dimensional and thus not amenable to direct numerical treatment. To obtain a tractable reformulation, we invoke Kantorovich duality for the Wasserstein metric \cite{kim2023distributional, beiglbock2012duality}, and arrive at
\begin{equation}
V_t(x)=x^{\top} Q x+\inf_u  \left[u^{\top} R u + \frac{1}{M} \sum_{i=1}^M \sup_{w \in \Omega} \{V_{t+1}(A_{t}x + B_{t}u + D_{t}w) - \lambda \|\widehat{w}_{t,i} - w\|^2 \}\right],
\label{value-func}
\end{equation}
which replaces the maximization over measures by a finite sum of pointwise optimization problems over $w$. We denote the empirical mean and covariance of the noise at stage $t$ by
\begin{equation}
\bar{w}_t =\mathbb{E}_{\widehat{\mathbb{P}}_M}\left[w_t \right] \mbox{ and } \Sigma^w_t = \mathbb{E}_{\widehat{\mathbb{P}}_M}\left[w_t w_t^{\top}\right].
\label{mu-cov}
\end{equation}
In the subsequent Lemma \ref{optimal-minimax}, we derive an explicit solution to the minimax problem associated with (\ref{w-penalty}) under the following condition on the penalty parameter.
\begin{assumption}
The penalty parameter satisfies $\lambda>\bar{\lambda}_t$ for all $t \geq 1$, where $\bar{\lambda}_t$ is the maximum eigenvalue of $D_t^{\top} P_t D_t$.
\label{lambda-ass}
\end{assumption}
This ensures that the resulting quadratic forms are strictly concave in the disturbance variable, which will be crucial for solving the inner maximization in closed form.
\begin{lemma}
Suppose that
$$
V_{t+1}(x)=x^{\top} P_{t+1} x  + 2 r_{t+1}^{\top} x + q_{t+1}
$$
for some nonnegative $P_{t+1} \in \mathbb{R}^{m\times m}$, $r_{t+1} \in \mathbb{R}^m$ and $q_{t+1} \in \mathbb{R}$. We further assume that the penalty parameter satisfies $\lambda>\bar{\lambda}_{t+1}$, where $\bar{\lambda}_{t+1}$ is the maximum eigenvalue of $D_t^{\top} P_{t+1} D_t$. 
Let $\Phi(w)= V_{t+1}(A_{t}x + B_{t}u + D_{t}w) - \lambda \|\widehat{w}_{t,i} - w\|^2$ and find $\frac{\partial \Phi(w)}{\partial w} = 0$. Then, the inner maximization problem $\sup_{w \in \Omega} \Phi(w)$ in (\ref{value-func}) has a unique maximizer $w_{t}^*= (w_{t,1}^*, \cdots, w_{t,M}^*)$, defined as
\begin{equation}
w_{t,i}^*= \Delta'_t (D_t^{\top} P_{t+1}(A_t x+B_t u)+ D_t^{\top} r_{t+1}+\lambda \widehat{w}_{t,i} ) .
\label{optimal-w}
\end{equation}
where $\Delta'_t = \left(\lambda \mathbf{I}_k - D_t^{\top}P_{t+1}D_t\right)^{-1}$. Substituting (\ref{mu-cov}) and (\ref{optimal-w}) into (\ref{value-func}), we can obtain the optimal value function with dynamic programming principle
$$
\begin{aligned}
V_t(x) &= x^{2}\left(Q + A_t P_{t+1} \Delta_t A_t  \right) + 2x \left(A_t \Delta_t r_{t+1} + \lambda A_t P_{t+1} D_t \Delta'_t \bar{w}_t \right)\\
&\quad + \inf_u \left\{u^{\top} (R +  B_t^{\top} P_{t+1} \Delta_t B_t )u + 2u^{\top} \left[(B_t^{\top}P_{t+1} \Delta_t  A_t)x +  B_t^{\top} \Delta_t r_{t+1} + \lambda B_t^{\top}  P_{t+1} D_t  \Delta'_t \bar{w}_t \right] \right\} \\
&\quad + r_{t+1} D_t \Delta'_t (D_t^{\top} r_{t+1} + 2 \lambda \bar{w}_t) + \operatorname{Tr}\left( \lambda^2 \Delta'_t \Sigma^w_t - \lambda \Sigma^w_t \right) + q_{t+1},
\end{aligned}
$$
where $\Delta_t = \mathbf{I}_m+D_t\Delta'_t D_t^{\top}P_t$. One can find that the outer minimization problem in (\ref{value-func}) has a unique minimizer 
$$
u^* = -K^*_t x - L^*_t,
$$
where
$$
\begin{aligned} 
K^*_t &=(R +  B_t^{\top} P_{t+1} \Delta_t B_t )^{-1}B_t^{\top}P_{t+1} \Delta_t  A_t, \\ 
L^*_t &= (R +  B_t^{\top} P_{t+1} \Delta_t B_t )^{-1}(B_t^{\top} \Delta_t r_{t+1} + \lambda B_t^{\top}  P_{t+1} D_t  \Delta'_t \bar{w}_t ).
\end{aligned}
$$
Thus the explicit derivation with this linear structure yields the following Riccati equation:
$$
\begin{aligned}
P_t & = Q+A_t P_{t+1} \Delta_t A_t - A_t \Delta_t P_{t+1} B_t (R_t+B_t^{\top} P_{t+1} \Delta_t B_t)^{-1} B_t^{\top} P_{t+1} \Delta_t A_t\\ 
r_t & =  A_t \Delta_t r_{t+1} + \lambda A_t P_{t+1} D_t \Delta'_t \bar{w}_t - A_t \Delta_t P_{t+1} B_t (R_t+B_t^{\top} P_{t+1} \Delta_t B_t)^{-1} \\
&\quad \cdot (B_t^{\top} \Delta_t r_{t+1} + \lambda B_t^{\top}  P_{t+1} D_t  \Delta'_t \bar{w}_t ) \\ 
q_t & =q_{t+1} + r_{t+1} D_t \Delta'_t (D_t^{\top} r_{t+1} + 2 \lambda \bar{w}_t) + \operatorname{Tr}\left( \lambda^2 \Delta'_t \Sigma^w_t - \lambda \Sigma^w_t \right)\\
&\quad - (B_t^{\top} \Delta_t r_{t+1} + \lambda B_t^{\top}  P_{t+1} D_t  \Delta'_t \bar{w}_t )^{\top} (R +  B_t^{\top} P_{t+1} \Delta_t B_t )^{-1} (B_t^{\top} \Delta_t r_{t+1} + \lambda B_t^{\top}  P_{t+1} D_t  \Delta'_t \bar{w}_t )
\end{aligned}
$$
with the terminal conditions $P_T=Q_T$, $r_T=\mathbf{0}_{m\times 1}$, and $q_T=0$. 
\label{optimal-minimax}
\end{lemma}
Suppose that Assumption \ref{lambda-ass} holds. Then, the matrices $P_t$ are well-defined and the value function can be expressed as
$$
V_t({x})={x}^{\top} P_t x +2 r_t^{\top} {x}+q_t, \quad t=0, \ldots, T
$$
Furthermore, the regularized problem (\ref{w-penalty}) in the finite horizon case has a unique optimal policy, defined as
$$
\pi_t^*({x})=-K^*_t {x}-L^*_t, \quad t=0, \ldots, T-1.
$$

\section{Gradient Dynamics and Convergence Analysis}
In the gradient analysis we treat the noise contribution via its empirical moments $\bar{w}_t$ and $\Sigma_t^w$ as fixed (data-driven) quantities and perform the minimization over $u$ while deferring the maximization over $w$. By the envelope theorem, since these moments correspond to the unique maximizer of the inner problem, their dependence on the control $K$ does not affect the gradient. Consequently,  $\bar{w}_t$ and $\Sigma_t^w$ act as constants in the Bellman expansion and do not introduce implicit derivatives in the cost gradient.

We represent an affine controller $u_t=-K_t x_t-L_t$ as a single linear map by augmenting the state with a constant 1, fixing the policy convention explicitly to avoid sign errors:
$$
u_t=-\widehat{K}_t \widehat{x}_t, \quad \widehat{K}_t=\begin{bmatrix}
K_t & L_t
\end{bmatrix}, \quad \widehat{x}_t=\begin{bmatrix}
x_t \\
1
\end{bmatrix}.
$$
Embed the original number into augmented matrices:
$$
\widehat{A}_t=\begin{bmatrix}
A_t & \textbf{0}_{m\times1} \\
\textbf{0}_{1\times m} & 1
\end{bmatrix}, \quad \widehat{B}_t=\begin{bmatrix}
B_t \\
\textbf{0}_{1\times n}
\end{bmatrix}, \quad \widehat{D}_t=\begin{bmatrix}
D_t \\
\mathbf{0}_{1\times k}
\end{bmatrix}, \quad \widehat{Q}_t=\begin{bmatrix}
Q_t & \mathbf{0}_{m\times 1} \\
\mathbf{0}_{1\times m} & 0
\end{bmatrix}.
$$
Since an admissible policy can be fully characterized by $\widehat{K}$, the cost of a policy $\widehat{K}$ can be correspondingly defined as
\begin{equation}
C(\widehat{K})=\mathbb{E}\left[\sum_{t=0}^{T-1}\left(\widehat{x}_t^{\top} \widehat{Q}_t \widehat{x}_t+u_t^{\top} R_t u_t\right)+\widehat{x}_T^{\top} \widehat{Q}_T \widehat{x}_T\right],
\label{widehat_K}
\end{equation}
subject to the dynamics
$$
\widehat{x}_{t+1} = \widehat{A}_t\widehat{x}_t + \widehat{B}_tu_t+\widehat{D}_tw_t. 
$$
Recall that $\widehat{K}^*$ is the optimal policy for the problem
$$
\widehat{K}^*=\underset{\widehat{K}}{\arg \min } C(\widehat{K}).
$$
\begin{assumption}[Initial State and Noise Process]
We make the following assumptions:
\begin{itemize}
    \item[1.] \textbf{Initial state:} The initial state satisfies $x_0 \sim \mathcal{D}$, where $\mathbb{E}\!\left[\widehat{x}_0\widehat{x}_0^{\top}\right]$ is positive definite. 
    \item[2.] \textbf{Noise process:} The sequence $\{w_t\}_{t=0}^{T-1}$ consists of independent and identically distributed (i.i.d.) random vectors that are independent of $x_0$. Their empirical mean and covariance are denoted by 
    \[
    \bar{w}_t = \mathbb{E}_{\widehat{\mathbb{P}}_M}[w_t], 
    \quad 
    \Sigma_t^w = \mathbb{E}_{\widehat{\mathbb{P}}_M}[w_t w_t^{\top}],
    \]
    where $\Sigma_t^w \succ 0$ for all $t=0,1,\dots,T-1$.
\end{itemize}
\label{as:3.1}
\end{assumption}

\medskip

We next introduce several quantities that will be used in our analysis. Specifically, we define $\underline{\sigma}_{X}$ as the uniform lower bound on the smallest singular values of the state covariance matrices $\mathbb{E}[\widehat{x}_t\widehat{x}_t^{\top}]$. Similarly, $\underline{\sigma}_{R}$ and $\underline{\sigma}_{Q}$ denote the smallest singular values of the control and state cost matrices, respectively:
\begin{subequations}
\begin{align}
\underline{\sigma}_{X} &= \min_{t} \; \sigma_{\min}\!\big(\mathbb{E}[\widehat{x}_t\widehat{x}_t^{\top}]\big), \label{sigmax} \\
\underline{\sigma}_{R} &= \min_{t} \; \sigma_{\min}(R_t), \label{sigmaR} \\
\underline{\sigma}_{Q} &= \min_{t} \; \sigma_{\min}(\widehat{Q}_t). \label{sigmaQ}
\end{align}
\end{subequations}
By Assumption~\ref{as:3.1}, it follows that $\underline{\sigma}_{R}>0$ and $\underline{\sigma}_{Q}>0$.

Finally, we introduce the notation $\mathcal{H}$ to denote the class of functions that are polynomial in the model parameters. When additional arguments are present, we write $\mathcal{H}(\cdot)$ to emphasize their dependence. The set of model parameters includes, but is not limited to, the system dimension $n$, the matrices $(\widehat{A}_t,\widehat{B}_t,\widehat{D}_t)$, the cost matrices $(\widehat{Q}_t,R_t)$, their associated norms and reciprocals, as well as the quantities $\underline{\sigma}_{X}$, $\underline{\sigma}_{R}$, $\underline{\sigma}_{Q}$, and $\mathbb{E}[\widehat{x}_0\widehat{x}_0^{\top}]$.

\subsection{Explicit Characterization of the Policy Gradient} 
We study an exact gradient descent scheme for computing the optimal feedback policy. The update rule is given by,
\begin{equation}
\widehat{K}_t^{n+1}=\widehat{K}_t^n-\eta \nabla_t C(\widehat{K}^n), \quad  \forall\ 0 \leq t \leq T-1,
\label{EGDM}
\end{equation}
where $n$ denotes the iteration index, $\eta>0$ is the step size, and $\nabla_t C(\widehat{K})=\frac{\partial C(\widehat{K})}{\partial \widehat{K}_t}$ is the gradient of the cost function with respect to $\widehat{K}_t$. We write $\nabla C(\widehat{K})=(\nabla_0 C(\widehat{K}), \cdots, \nabla_{T-1} C(\widehat{K}))$ for compactness. Let $\{\widehat{x}_t\}_{t=0}^T$ denote the trajectory under the current policy $\widehat{K}$. Define the state mean and covariance at time $t$ as 
\begin{equation}
\mu_t = \mathbb{E}\left[\widehat{x}_t^{\top}\right] \mbox{ and } \Sigma_t=\mathbb{E}\left[\widehat{x}_t \widehat{x}_t^{\top}\right], \quad t=0,1, \cdots, T,
\end{equation}
We further define the aggregated quantities over the horizon: a matrix $\Sigma_{\widehat{K}}$ as the sum of $\Sigma_t$,
\begin{equation}
\Sigma_{\widehat{K}}=\sum_{t=0}^T \Sigma_t=\mathbb{E}\left[\sum_{t=0}^T \widehat{x}_t \widehat{x}_t^{\top}\right],
\end{equation}
and a matrix $\mu_{\widehat{K}}$ as the sum of $\mu_t$,
\begin{equation}
\mu_{\widehat{K}}=\sum_{t=0}^T \mu_t=\mathbb{E}\left[\sum_{t=0}^T \widehat{x}_t^{\top}\right].
\end{equation}
In the finite time horizon setting, define $\widehat{P}_t^{\widehat{K}}$, $\widehat{r}_t^{\widehat{K}}$ and $\widehat{q}_t^{\widehat{K}}$ as the solution to the backward recursions
\begin{equation}
\begin{aligned}
\widehat{P}_t^{\widehat{K}} &=\widehat{Q}_t+\widehat{K}_t^{\top} R_t \widehat{K}_t+ (\widehat{A}_t-\widehat{B}_t \widehat{K}_t)^{\top} \widehat{P}_{t+1}^{\widehat{K}}  (\widehat{A}_t-\widehat{B}_t \widehat{K}_t), \\
\widehat{r}_t^{\widehat{K}} &= (\widehat{A}_t-\widehat{B}_t \widehat{K}_t)^{\top} \widehat{r}_{t+1}^{\widehat{K}}+ (\widehat{A}_t-\widehat{B}_t \widehat{K}_t)^{\top} \widehat{P}_{t+1}^{\widehat{K}} \widehat{D}_t \bar{w}_t, \\
\widehat{q}_t^{\widehat{K}} &= \widehat{q}_{t+1}^{\widehat{K}}+2 (\widehat{r}_{t+1}^{\widehat{K}})^{\top} \widehat{D}_t \bar{w}_t+\operatorname{Tr}(\widehat{D}_t^{\top} \widehat{P}_{t+1}^{\widehat{K}} \widehat{D}_t  \Sigma_t^w). \\
\end{aligned}
\label{Ptk}
\end{equation}
for $t=0,1, \cdots, T-1$ with terminal conditions
$$
\widehat{P}_T^{\widehat{K}} = \begin{bmatrix} P_T & \mathbf{0}_{m\times 1} \\ \mathbf{0}_{1\times m}& 0 \end{bmatrix}, \quad \widehat{r}_T^{\widehat{K}} = \begin{bmatrix} r_T \\ 0 \end{bmatrix}, \quad \widehat{q}_T^{\widehat{K}} = q_T,
$$
where $P_T = Q_T$, $r_T=\mathbf{0}_{m\times 1}$, and $q_T=0$. To simplify, we remove $\widehat{K}$-superscript when there is no confusion. Then the cost of $\widehat{K}$ can be rewritten as
$$
C(\widehat{K})=\mathbb{E}_{x_0 \sim \mathcal{D}}\left[\widehat{x}_0^{\top} \widehat{P}_0 \widehat{x}+ 2\widehat{r}_0^{\top}\widehat{x}_0 + \widehat{q}_0
\right].
$$
To see this,
$$
\begin{aligned}
&\quad \  \mathbb{E}\left[\widehat{x}_0^{\top} \widehat{P}_0 \widehat{x}+ 2\widehat{r}_0^{\top}\widehat{x}_0 + \widehat{q}_0 \right] \\ %+L_0
& =\mathbb{E}\left[\widehat{x}_0^{\top} \widehat{Q}_0 \widehat{x}_0 + \widehat{x}_0^{\top} \widehat{K}_0^{\top} R_0 \widehat{K}_0  \widehat{x} + \widehat{x}_0^{\top} (\widehat{A}_0-\widehat{B}_0 \widehat{K}_0)^{\top} \widehat{P}_1 (\widehat{A}_0-\widehat{B}_0 \widehat{K}_0) \widehat{x}_0 \right. \\
& \quad \left. + 2\widehat{r}_{1}^{\top}(\widehat{A}_t-\widehat{B}_t \widehat{K}_t)\widehat{x}_0 + 2 \bar{w}_t^{\top}\widehat{D}_t^{\top} \widehat{P}_1(\widehat{A}_t-\widehat{B}_t \widehat{K}_t)\widehat{x}_0   \right.  \\
&\quad \left. + \widehat{q}_{1}+2  \widehat{r}_{1}^{\top}  \widehat{D}_0 \bar{w}_0+ \operatorname{Tr}( \widehat{D}_0^{\top}  \widehat{P}_{1}  \widehat{D}_0 \Sigma_0^w)  \right] \\
& =\mathbb{E}\left[\widehat{x}_0^{\top} \widehat{Q}_0 \widehat{x}_0+u_0^{\top} R_0 u_0+\widehat{x}_1^{\top} \widehat{P}_1 \widehat{x}_1 + 2\widehat{r}_1^{\top}\widehat{x}_1 + \widehat{q}_1\right]\\
& =\mathbb{E}\left[\sum_{t=0}^{T-1}\left(\widehat{x}_t^{\top} \widehat{Q}_t \widehat{x}_t+u_t^{\top} R_t u_t\right)+\widehat{x}_T^{\top} Q_T \widehat{x}_T \right] .
\end{aligned}
$$

\begin{lemma}[Gradient Representation]
Define for each $t = 0,1,\dots, T-1$,
\begin{equation}
\begin{aligned}
E_{t}&= (R_t+\widehat{B}_{t}^{\top} \widehat{P}_{t+1} \widehat{B}_t) \widehat{K}_t-\widehat{B}_{t}^{\top} \widehat{P}_{t+1} \widehat{A}_t, \\
F_t &= \widehat{B}_t^{\top} \widehat{P}_{t+1} \widehat{D}_t \bar{w}_t+\widehat{B}_t^{\top} \widehat{r}_{t+1}.
\end{aligned}
\label{Et}
\end{equation}
Then we have the following representation of the gradient term with respect to the policy $\widehat{K}_t$, 
$$
\nabla_t C(\widehat{K}) =2 E_t \Sigma_t - 2F_t \mu_t.
$$
\label{le3.5}  
\end{lemma}
\begin{proof}%[\textbf{Proof of Lemma \ref{le3.5}}]
We illustrate the derivation for $t=0$ and the argument extends identically to all $t=0,1,\dots,T-1$. By definition,
\begin{align*}
\nabla_0 C(\widehat{K}) &= \frac{\partial C(\widehat{K})}{\partial \widehat{K}_0} \\
&= \mathbb{E}\left[2 R_0 \widehat{K}_0 \widehat{x}_0 \widehat{x}_0^{\top} - 2 \widehat{B}_{0}^{\top} \widehat{P}_1(\widehat{A}_0-\widehat{B}_0 \widehat{K}_0) \widehat{x}_0 \widehat{x}_0^{\top} - 2 \widehat{B}_0^{\top} \widehat{r}_1  \widehat{x}_0^{\top} - 2\widehat{B}_0^{\top} \widehat{P}_1^{\top}\widehat{D}_0 \bar{w}_0 \widehat{x}_0^{\top} \right] \\
&=2 E_0 \mathbb{E}\left[\widehat{x}_0 \widehat{x}_0^{\top}\right] - 2 F_0 \mathbb{E}\left[\widehat{x}_0^{\top}\right] \\
&=2 E_0 \Sigma_0 - 2F_0 \mu_0. 
\end{align*}
Similarly, $\forall t=0,1, \cdots, T-1$,
$$
\begin{aligned}
\nabla_t C(\widehat{K})&=2\left((R_t+\widehat{B}_{t}^{\top} \widehat{P}_{t+1} \widehat{B}_t) \widehat{K}_t-\widehat{B}_{t}^{\top} \widehat{P}_{t+1} \widehat{A}_t \right) \mathbb{E}\left[\widehat{x}_t \widehat{x}_t^{\top}\right]\\
& \quad - \left(\widehat{B}_t^{\top} \widehat{P}_{t+1} \widehat{D}_t \bar{w}_t+\widehat{B}_t^{\top} \widehat{r}_{t+1} \right) \mathbb{E}\left[ \widehat{x}_t^{\top} \right]\\
&=2 E_t \mathbb{E}\left[\widehat{x}_t\widehat{x}_t^{\top}\right] - 2F_t \mathbb{E}\left[\widehat{x}_t^{\top}\right] \\
&=2 E_t \Sigma_t - 2F_t\mu_t,
\end{aligned}
$$
where the expectation $\mathbb{E}$ is taken with respect to both initial distribution $x_0 \sim \mathcal{D}$ and noises $w$.
\end{proof}
In the following lemma, we establish the analytical form of the advantage function, which serves as the foundation for subsequent results. Building upon this, Lemma \ref{le3.6} shows that for any given policy $\widehat{K}$, the optimality gap $C(\widehat{K}) - C(\widehat{K}^*)$ is bounded by the sum of the magnitude of the gradient $\nabla_t C(\widehat{K})$ for $t=0,1, \cdots, T-1$. 

Let us start with a useful result for the value function. Define the value function $V_{\widehat{K}}(\widehat{x}, \tau)$ for $\tau=0,1, \cdots, T-1$, as
$$
V_{\widehat{K}}(\widehat{x}, \tau)=\mathbb{E}\left[\sum_{t=\tau}^{T-1}\left(\widehat{x}_t^{\top} \widehat{Q}_t \widehat{x}_t + u_t^{\top} R_t u_t\right)+\widehat{x}_T^{\top} \widehat{Q}_T \widehat{x}_t \bigg| \widehat{x}_\tau=x\right]=\widehat{x}^{\top} \widehat{P}_\tau \widehat{x}_t + 2\widehat{r}_{\tau}\widehat{x} + \widehat{q}_{\tau},
$$
with terminal condition
$$
V_{\widehat{K}}(\widehat{x}, T)=\widehat{x}^{\top} \widehat{Q}_T \widehat{x}.
$$
We then define the $Q$-function, $Q_{\widehat{K}}(\widehat{x}, u, \tau)$ for $\tau=0,1, \cdots, T-1$ as
$$
Q_{\widehat{K}}(\widehat{x}, u, \tau)=\widehat{x}^{\top} \widehat{Q}_\tau \widehat{x}_\tau +u^{\top} R_\tau u+\mathbb{E}_{w_\tau}\left[V_{\widehat{K}}\left(\widehat{A}_{\tau} \widehat{x}+ \widehat{B}_{\tau} u + \widehat{D}_\tau w_\tau, \tau+1\right)\right], 
$$
and the advantage function
$$
A_{\widehat{K}}(\widehat{x}, u, \tau)=Q_{\widehat{K}}(\widehat{x}, u, \tau)-V_{\widehat{K}}(\widehat{x}, \tau) .
$$
Note that $C(\widehat{K})=\mathbb{E}_{x_0 \sim \mathcal{D}}\left[V(\widehat{x}_0, 0)\right]$. Then we can write the difference of value functions between $\widehat{K}$ and $\widehat{K}'$ in terms of advantage functions in Lemma \ref{leC.1}. 

\begin{lemma}
Assume $\widehat{K}$ and $\widehat{K}'$ have finite costs. Denote $\left\{\widehat{x}'_t\right\}_{t=0}^T$ and $\left\{u'_t\right\}_{t=0}^{T-1}$ as the state and control sequences of a single trajectory generated by $\widehat{K}'$ starting from $\widehat{x}'_0=\widehat{x}_0=\widehat{x}$, then
\begin{equation}
V_{\widehat{K}'}(\widehat{x}, 0)-V_{\widehat{K}}(\widehat{x}, 0)=\mathbb{E}_{{w}}\left[\sum_{t=0}^{T-1} A_{\widehat{K}}\left(\widehat{x}'_t, u'_t, t\right)\right]
\end{equation}
and $A_{\widehat{K}}(\widehat{x},-\widehat{K}'_\tau \widehat{x}, \tau)=2 \widehat{x}^{\top}(\widehat{K}_\tau^{\prime}-\widehat{K}_\tau)^{\top} E_\tau \widehat{x} +\widehat{x}^{\top}(\widehat{K}_\tau^{\prime}-\widehat{K}_\tau)^{\top}(R_\tau+\widehat{B}_{\tau}^{\top} \widehat{P}_{\tau+1} \widehat{B}_{\tau})(\widehat{K}_\tau^{\prime}-\widehat{K}_\tau)\widehat{x}-2\widehat{x}^{\top}(\widehat{K}_\tau^{\prime}-\widehat{K}_\tau)^{\top}F_\tau$, where $E_\tau$ and $F_\tau $ are defined in (\ref{Et}).
\label{leC.1}
\end{lemma}
\begin{proof}%[\textbf{Proof of Lemma \ref{leC.1}}]
Denote by $c'_t(\widehat{x})$ the cost generated by $\widehat{K}'$ with a single trajectory starting from $\widehat{x}'_0=\widehat{x}_0=\widehat{x}$. That is, $c'_t(\widehat{x})=(\widehat{x}'_t)^{\top} \widehat{Q}_t \widehat{x}'_t +(u'_t)^{\top} R_t u'_t, t=0,1, \cdots, T-1$, and $c'_T(\widehat{x})=(\widehat{x}'_T)^{\top} Q_T \widehat{x}'_T $, with $u'_t=-\widehat{K}'_t \widehat{x}'_t$, $\widehat{x}'_{t+1}=\widehat{A}_t \widehat{x}'_t+ \widehat{B}_{t} u'_t+\widehat{D}_tw_t$, $\widehat{x}_0=\widehat{x}$. Therefore,
$$
\begin{aligned}
V_{\widehat{K}'}(\widehat{x}, 0)-V_{\widehat{K}}(\widehat{x}, 0) & =\mathbb{E}_{{w}}\left[\sum_{t=0}^T c'_t(\widehat{x})\right]-V_{\widehat{K}}(\widehat{x}, 0)\\
&=\mathbb{E}_{{w}}\left[\sum_{t=0}^T\left(c'_t(\widehat{x})+V_{\widehat{K}}\left(\widehat{x}'_t, t\right)-V_{\widehat{K}}\left(\widehat{x}'_t, t\right)\right)\right]-V_{\widehat{K}}(\widehat{x}, 0) \\
& =\mathbb{E}_{{w}}\left[\sum_{t=0}^{T-1}\left(c'_t(\widehat{x})+V_{\widehat{K}}\left(\widehat{x}'_{t+1}, t+1\right)-V_{\widehat{K}}\left(\widehat{x}'_t, t\right)\right)\right] \\
& =\mathbb{E}_{{w}}\left[\sum_{t=0}^{T-1}\left(Q_{\widehat{K}}\left(\widehat{x}'_t, u'_t, t\right)-V_{\widehat{K}}\left(\widehat{x}'_t, t\right)\right) \bigg| \widehat{x}_0=\widehat{x}\right] \\
&=\mathbb{E}_{{w}}\left[\sum_{t=0}^{T-1} A_{\widehat{K}}\left(\widehat{x}'_t, u'_t, t\right) \bigg| \widehat{x}_0=\widehat{x}\right],
\end{aligned}
$$
where the third equality holds since $c'_T(\widehat{x})=V_{\widehat{K}}\left(\widehat{x}'_t, T\right)$ with the same single trajectory. For $u=-\widehat{K}'_\tau \widehat{x}$, we have
\begin{equation}
\begin{aligned}
&\quad \ A_{\widehat{K}}(\widehat{x},-\widehat{K}_\tau^{\prime} \widehat{x}, \tau)\\
&= Q_{\widehat{K}}(\widehat{x},-\widehat{K}_\tau^{\prime} \widehat{x}, \tau)-V_{\widehat{K}}(\widehat{x}, \tau) \\
&= \widehat{x}^{\top}(\widehat{Q}_\tau+(\widehat{K}_\tau^{\prime})^{\top} R_\tau \widehat{K}_\tau^{\prime}) \widehat{x} +\mathbb{E}_{w_\tau}\left[V_{\widehat{K}}((\widehat{A}_\tau-\widehat{B}_{\tau} \widehat{K}_\tau^{\prime}) \widehat{x} + \widehat{D}_\tau w_\tau, \tau+1)\right]-V_{\widehat{K}}(\widehat{x}, \tau) \\
&=\widehat{x}^{\top}(\widehat{Q}_\tau+(\widehat{K}_\tau^{\prime})^{\top} R_\tau \widehat{K}_\tau^{\prime})\widehat{x} + \mathbb{E} \Bigl[\widehat{x}^{\top}(\widehat{A}_\tau-\widehat{B}_{\tau} \widehat{K}'_\tau)^{\top}\widehat{P}_{\tau+1}(\widehat{A}_\tau-\widehat{B}_{\tau} \widehat{K}'_\tau) \widehat{x} + w_\tau^{\top} \widehat{D}_\tau^{\top} \widehat{P}_{\tau+1}\widehat{D}_\tau w_\tau \\
& \quad + 2  w_\tau^{\top} \widehat{D}_\tau^{\top}\widehat{P}_{\tau+1}(\widehat{A}_\tau - \widehat{B}_\tau \widehat{K}'_\tau)\widehat{x}+ 2\widehat{r}_{\tau + 1}^{\top}((\widehat{A}_\tau-\widehat{B}_{\tau} \widehat{K}'_\tau)\widehat{x}+\widehat{D}_\tau w_\tau) +  \widehat{q}_{\tau+1} \Bigr]\\
&\quad- \mathbb{E}\left[\widehat{x}^{\top} \widehat{P}_\tau \widehat{x} + 2\widehat{r}_\tau^{\top} \widehat{x} + \widehat{q}_\tau\right] \\
&=\widehat{x}^{\top}(\widehat{Q}_\tau+(\widehat{K}_\tau^{\prime})^{\top} R_\tau \widehat{K}_\tau^{\prime})\widehat{x} + \widehat{x}^{\top}(\widehat{A}_\tau-\widehat{B}_{\tau} \widehat{K}'_\tau)^{\top}\widehat{P}_{\tau+1}(\widehat{A}_\tau-\widehat{B}_{\tau} \widehat{K}'_\tau) \widehat{x} \\
& \quad +  2  w_\tau^{\top} \widehat{D}_\tau^{\top}\widehat{P}_{\tau+1}(\widehat{A}_\tau - \widehat{B}_\tau \widehat{K}'_\tau)\widehat{x} + 2\widehat{r}_{\tau + 1}^{\top}(\widehat{A}_\tau-\widehat{B}_{\tau} \widehat{K}'_\tau)\widehat{x} - (\widehat{x}^{\top} \widehat{P}_\tau \widehat{x} + 2\widehat{r}_\tau^{\top} \widehat{x}) \\
&= \widehat{x}^{\top}(\widehat{Q}_\tau+(\widehat{K}_\tau^{\prime}-\widehat{K}_\tau+\widehat{K}_\tau)^{\top} R_\tau(\widehat{K}_\tau^{\prime}-\widehat{K}_\tau+\widehat{K}_\tau))\widehat{x} \\
& \quad +\widehat{x}^{\top}(\widehat{A}_\tau- \widehat{B}_{\tau} \widehat{K}_\tau- \widehat{B}_{\tau} (\widehat{K}'_\tau-\widehat{K}_\tau))^{\top} \widehat{P}_{\tau+1} (\widehat{A}_\tau- \widehat{B}_{\tau} \widehat{K}_\tau- \widehat{B}_{\tau} (\widehat{K}'_\tau-\widehat{K}_\tau)) \widehat{x}\\
&\quad + 2\bar{w}_\tau^{\top} \widehat{D}_\tau^{\top} \widehat{P}_{\tau+1}(\widehat{A}_\tau- \widehat{B}_{\tau} \widehat{K}_\tau- \widehat{B}_{\tau} (\widehat{K}'_\tau-\widehat{K}_\tau))\widehat{x} + 2\widehat{r}_{\tau+1}^{\top}(\widehat{A}_\tau- \widehat{B}_{\tau} \widehat{K}_\tau- \widehat{B}_{\tau} (\widehat{K}'_\tau-\widehat{K}_\tau))\widehat{x}\\
& \quad -\widehat{x}^{\top}(\widehat{Q}_\tau+\widehat{K}_\tau^{\top} R_\tau \widehat{K}_\tau+ (\widehat{A}_\tau- \widehat{B}_{\tau} \widehat{K}_\tau)^{\top} \widehat{P}_{\tau+1}(\widehat{A}_\tau- \widehat{B}_{\tau} \widehat{K}_\tau))\widehat{x}  \\
& \quad - 2\widehat{r}_{\tau+1}^{\top}(\widehat{A}_\tau -\widehat{B}_\tau \widehat{K}_\tau)\widehat{x} - 2\bar{w}_\tau^{\top}\widehat{D}_\tau^{\top} \widehat{P}_{\tau+1}(\widehat{A}_\tau-\widehat{B}_\tau\widehat{K}_\tau)\widehat{x}  \\
&= 2 \widehat{x}^{\top}(\widehat{K}'_\tau-\widehat{K}_\tau)^{\top}((R_\tau+\widehat{B}_{\tau}^{\top} \widehat{P}_{\tau+1} \widehat{B}_{\tau}) \widehat{K}_\tau-\widehat{B}_{\tau}^{\top} \widehat{P}_{\tau+1} \widehat{A}_{\tau})\widehat{x} \\
&\quad +\widehat{x}^{\top}(\widehat{K}_\tau^{\prime}-\widehat{K}_\tau)^{\top}(R_\tau+\widehat{B}_{\tau}^{\top} \widehat{P}_{\tau+1} \widehat{B}_{\tau})(\widehat{K}_\tau^{\prime}-\widehat{K}_\tau)\widehat{x}\\
&\quad - 2 \widehat{x}^{\top} (\widehat{K}'_\tau-\widehat{K}_\tau)^{\top} \widehat{B}_{\tau}^{\top} (\widehat{P}_{\tau+1}  \widehat{D}_\tau \bar{w}_\tau +\widehat{r}_{\tau+1}).
\end{aligned}
\label{C.2}
\end{equation}
\end{proof}

\begin{lemma}
Let $\widehat{K}^*$ be an optimal policy and $C(\widehat{K})$ be finite, then
\begin{align*}
(\underline{\sigma}_{X}+1) \sum_{t=0}^{T-1} \frac{\operatorname{Tr}(E_t^{\top} E_t)+F_t^{\top}F_t}{\|R_t+\widehat{B}_{t}^{\top} \widehat{P}_{t+1} \widehat{B}_t \|} \leq C(\widehat{K})-C(\widehat{K}^{*}) \leq \frac{\|\Sigma_{\widehat{K}^{*}}\|+1}{4 {\underline{\sigma}_{{X}}}^2 \underline{\sigma}_{R}} \sum_{t=0}^{T-1} \left(\nabla_t C(\widehat{K})^{\top} \nabla_t C(\widehat{K})\right),
\end{align*}
where $\underline{\sigma}_{X}$ and $\underline{\sigma}_{{Q}}$ are defined in (\ref{sigmax}) and (\ref{sigmaR}).
\label{le3.6}
\end{lemma}
\begin{proof}%[\textbf{Proof of Lemma \ref{le3.6}}]
First for any $\widehat{K}_\tau^{\prime}$, from (\ref{C.2}),
\begin{equation}
\begin{aligned}
& \quad \ A_{\widehat{K}}(\widehat{x},-\widehat{K}_\tau^{\prime} \widehat{x}, \tau) \\
&= Q_{\widehat{K}}(\widehat{x},-\widehat{K}_\tau^{\prime} \widehat{x}, \tau)-V_{\widehat{K}}(\widehat{x}, \tau) \\
&= 2\widehat{x}^{\top}(\widehat{K}_\tau^{\prime}-\widehat{K}_\tau)^{\top} (E_\tau\widehat{x} -F_\tau) + \widehat{x}^{\top}(\widehat{K}_\tau^{\prime}-\widehat{K}_\tau)^{\top}(R_\tau+\widehat{B}_{\tau}^{\top} \widehat{P}_{\tau+1} \widehat{B}_\tau)(\widehat{K}_\tau^{\prime}-\widehat{K}_\tau)\widehat{x}\\
&= ((\widehat{K}_\tau^{\prime}-\widehat{K}_\tau)\widehat{x} + (R_\tau+\widehat{B}_{\tau}^{\top} \widehat{P}_{\tau+1} \widehat{B}_\tau)^{-1}(E_\tau\widehat{x} -F_\tau))^{\top}(R_\tau+\widehat{B}_{\tau}^{\top} \widehat{P}_{\tau+1} \widehat{B}_\tau)\\
&\quad \cdot ((\widehat{K}_\tau^{\prime}-\widehat{K}_\tau)\widehat{x} + (R_\tau+\widehat{B}_{\tau}^{\top} \widehat{P}_{\tau+1} \widehat{B}_\tau)^{-1}(E_\tau\widehat{x} -F_\tau)) \\
&\quad - (E_\tau\widehat{x} -F_\tau)^{\top}(R_\tau+\widehat{B}_{\tau}^{\top} \widehat{P}_{\tau+1} \widehat{B}_\tau)^{-1}(E_\tau\widehat{x} -F_\tau) \\
&\geq -\operatorname{Tr}(\widetilde{x}\widetilde{x}^{\top} \widetilde{E}_\tau^{\top}(R_\tau+\widehat{B}_\tau^{\top} \widehat{P}_{\tau+1} \widehat{B}_\tau)^{-1} \widetilde{E}_\tau),
\end{aligned}
\label{C.3}
\end{equation}
where $\widetilde{x}_\tau=\begin{bmatrix}\widehat{x}_\tau \\ 1\end{bmatrix}$ and $\widetilde{E}_\tau = \begin{bmatrix} E_\tau & -F_\tau \end{bmatrix}$. The equality holds when $\widehat{K}_\tau^{\prime}=\widehat{K}_\tau-(R_\tau+\widehat{B}_{\tau}^{\top} \widehat{P}_{\tau+1} \widehat{B}_{\tau})^{-1} E_\tau$ and $(R_\tau+\widehat{B}_{\tau}^{\top} \widehat{P}_{\tau+1} \widehat{B}_\tau)^{-1}F_\tau =0$.
Then,
$$
\begin{aligned}
C(\widehat{K})-C(\widehat{K}^*) & = -\mathbb{E} \left[ \sum_{t=0}^{T-1} A_{K}\left(\widehat{x}_t^*, u_t^*, t\right)\right] \\
&\leq \mathbb{E} \left[ \sum_{t=0}^{T-1} \operatorname{Tr}\left(\widetilde{x}_t^* (\widetilde{x}^*_t)^{\top} \widetilde{E}_t^{\top}(R_t+\widehat{B}_{t}^{\top} \widehat{P}_{t+1} \widehat{B}_t)^{-1} \widetilde{E}_t\right) \right]\\
& \leq\left(\|\Sigma_{\widehat{K}^*}\|+1\right) \sum_{t=0}^{T-1} \operatorname{Tr}\left(\widetilde{E}_t^{\top}(R_t+\widehat{B}_{t}^{\top} \widehat{P}_{t+1} \widehat{B}_t)^{-1} \widetilde{E}_t\right) \\
&\leq \frac{\|\Sigma_{\widehat{K}^*}\|+1}{\underline{\sigma}_{R}} \sum_{t=0}^{T-1} \operatorname{Tr}\left(\widetilde{E}_t^{\top} \widetilde{E}_t\right) \\
& \leq \frac{\|\Sigma_{\widehat{K}^*}\|+1}{4 \underline{\sigma}_{X}^2 \underline{\sigma}_{R}} \sum_{t=0}^{T-1} \operatorname{Tr}\left(\nabla_t C(\widehat{K})^{\top} \nabla_t C(\widehat{K})\right),
\end{aligned}
$$
where $\underline{\sigma}_{X}$ is defined in (\ref{sigmax}) and $\underline{\sigma}_{R}$ is defined in (\ref{sigmaR}). For the lower bound, consider $\widehat{K}_\tau^{\prime}=\widehat{K}_\tau-(R_\tau+\widehat{B}_{\tau}^{\top} \widehat{P}_{\tau+1} \widehat{B}_{\tau})^{-1} E_\tau$ and $(R_\tau+\widehat{B}_{\tau}^{\top} \widehat{P}_{\tau+1} \widehat{B}_\tau)^{-1}F_\tau =0$ where the equality holds in (\ref{C.3}). Using $C(K^*) \leq C(\widehat{K}^{\prime})$,
\begin{equation}
\begin{aligned}
C(\widehat{K})-C(\widehat{K}^*) & \geq C(\widehat{K})-C(\widehat{K}^{\prime}) \\
&=-\mathbb{E} \left[ \sum_{t=0}^{T-1} A_{\widehat{K}}\left(\widehat{x}_t^{\prime}, u_t^{\prime}, t\right) \right] \\
&=\mathbb{E} \left[ \sum_{t=0}^{T-1} \operatorname{Tr}\left(\widetilde{x}'_t (\widetilde{x}'_t)^{\top} \widetilde{E}_t^{\top}(R_t+\widehat{B}_{t}^{\top} \widehat{P}_{t+1} \widehat{B}_t)^{-1} \widetilde{E}_t\right) \right]\\
& \geq (\underline{\sigma}_{X}+1) \sum_{t=0}^{T-1} \frac{\operatorname{Tr}(E_t^{\top} E_t)+F_t^{\top}F}{\|R_t+\widehat{B}_{t}^{\top} \widehat{P}_{t+1} \widehat{B}_t\|}.
\end{aligned}
\end{equation}
\end{proof}
Having obtained the explicit policy gradient, we next investigate its structural properties. Lemmas \ref{le3.7} demonstrate an ``almost-smoothness’’ condition of the cost landscape, which shows that when $\widehat{K}^{\prime}$ is sufficiently close to $\widehat{K}$, $C(\widehat{K}^{\prime})-C(\widehat{K})$ is bounded by the sum of the first and second order terms in $\widehat{K}-\widehat{K}^{\prime}$. 
\begin{lemma}{\textup{(“Almost Smoothness”)}.}
Let $\left\{\widehat{x}_t^{\prime}\right\}$ be the sequence of states for a single trajectory generated by $\widehat{K}^{\prime}$ starting from $\widehat{x}_0^{\prime}=\widehat{x}_0$. Then, $C(\widehat{K})$ satisfies
$$
\begin{aligned}
&\quad \  C(\widehat{K}^{\prime})-C(\widehat{K}) \\
& = \sum_{t=0}^{T-1}\left[2\operatorname{Tr}(\Sigma_t^{\prime}(\widehat{K}_t^{\prime}-\widehat{K}_t)^{\top} E_t) + \operatorname{Tr}(\Sigma_t^{\prime}(\widehat{K}_t^{\prime}-\widehat{K}_t)^{\top}(R_t+\widehat{B}_{t}^{\top} \widehat{P}_{t+1} \widehat{B}_t)(\widehat{K}_t^{\prime}-\widehat{K}_t)) - 2\mu_t^{\prime}(\widehat{K}_t^{\prime}-\widehat{K}_t)^{\top}F_t\right],
\end{aligned}
$$
where $\Sigma_t^{\prime}=\mathbb{E}\left[\widehat{x}_t^{\prime}(\widehat{x}'_t)^{\top}\right]$ and $\mu'_t = \mathbb{E}\left[(\widehat{x}'_t)^{\top}\right]$.
\label{le3.7}
\end{lemma}
\begin{proof}%[\textbf{Proof of Lemma \ref{le3.7}}]
By Lemma \ref{leC.1} we have
$$
\begin{aligned}
&\quad \ C(\widehat{K}^{\prime})-C(\widehat{K}) \\
& =\mathbb{E}\left[\sum_{t=0}^{T-1} A_{\widehat{K}}(\widehat{x}_t^{\prime},-\widehat{K}_t^{\prime} \widehat{x}_t^{\prime}, t)\right] \\
& =\sum_{t=0}^{T-1}\left[2\operatorname{Tr}(\Sigma_t^{\prime}(\widehat{K}_t^{\prime}-\widehat{K}_t)^{\top} E_t) + \operatorname{Tr}(\Sigma_t^{\prime}(\widehat{K}_t^{\prime}-\widehat{K}_t)^{\top}(R_t+\widehat{B}_{t}^{\top} \widehat{P}_{t+1} \widehat{B}_t)(\widehat{K}_t^{\prime}-\widehat{K}_t)) - 2\mu_t^{\prime}(\widehat{K}_t^{\prime}-\widehat{K}_t)^{\top}F_t\right].
\end{aligned}
$$
\end{proof}
The first term and third term can be added as $\operatorname{Tr}((\widehat{K}_t-\widehat{K}'_t)\nabla_tC(\widehat{K}))$ by Lemma \ref{le3.5}, while the second term is the second order of $\widehat{K}_t-\widehat{K}'_t$. We further bound $P_t$ and $\Sigma_{\widehat{K}}$, which is provided below in Lemma \ref{le3.8}.
\begin{lemma}
For every $t=0,1, \cdots, T$ and for the aggregated quantities we have
$$
\left\|P_t\right\| \leq \frac{C(\widehat{K})}{\underline{\sigma}_{X}}, \quad \|\Sigma_{\widehat{K}}\| \leq \frac{C(\widehat{K})}{\underline{\sigma}_{Q}}, \quad \|\mu_{\widehat{K}}\| \leq \sqrt{\frac{T C(\widehat{K})}{\underline{\sigma}_Q}},
$$
where $\underline{\sigma}_{X}$ and $\underline{\sigma}_{{Q}}$ are defined in (\ref{sigmax}) and (\ref{sigmaQ}).
\label{le3.8}
\end{lemma}
\begin{proof}%[\textbf{Proof of Lemma \ref{le3.8}}]
For $t=0,1, \cdots, T$,
\begin{itemize}
\item[(i)] Bound on $\|\widehat{P}_t\|$,
$$
C(\widehat{K}) \geq \mathbb{E}\left[\widehat{x}_t^{\top} \widehat{P}_t\widehat{x}_t\right] \geq\|\widehat{P}_t\| \left(\mathbb{E}\left[x_t^{2}\right]\right) \geq \underline{\sigma}_{{X}}\|\widehat{P}_t\|.
$$
\item[(ii)] Bound on $\|\Sigma_{\widehat{K}}\|$,
$$
C(\widehat{K})=\sum_{t=0}^{T-1} \operatorname{Tr}\left(\mathbb{E}[\widehat{x}_t\widehat{x}_t^{\top}](\widehat{Q}_t+\widehat{K}_t^{\top} R_t \widehat{K}_t)\right)+\operatorname{Tr}\left(\mathbb{E}[\widehat{x}_T\widehat{x}_T^{\top}]\widehat{Q}_T\right)  \geq \underline{\sigma}_{Q}\|\Sigma_{\widehat{K}}\|.
$$
Rearranging yields the second inequality.
\item[(iii)] Bound on $\|\mu_{\widehat{K}}\|$,
First extract the mean contribution from the cost. Note
$$
\Sigma_t=\operatorname{Cov}\left(\widehat{x}_t\right)+\mu_t^{\top}\mu_t  \succeq \mu_t^{\top}\mu_t.
$$
So for each $t$,
$$
\operatorname{Tr}\left(\mathbb{E}[\widehat{x}_t \widehat{x}_t^{\top}] \widehat{Q}_t\right) \geq \mu_t \widehat{Q}_t \mu_t^{\top} \geq \underline{\sigma}_Q\left\|\mu_t\right\|^2
$$
Summing over $t=0, \ldots, T-1$ gives
$$
C(\widehat{K}) \geq \underline{\sigma}_Q \sum_{t=0}^{T-1}\left\|\mu_t\right\|^2 \quad \Longrightarrow \quad \sum_{t=0}^{T-1}\left\|\mu_t\right\|^2 \leq \frac{C(\widehat{K})}{\underline{\sigma}_Q} .
$$
Now apply vector Cauchy-Schwarz (or the inequality $\left\|\sum_t v_t\right\|^2 \leq(T+1) \sum_t\left\|v_t\right\|^2$ ):
$$
\|\mu_{\widehat{K}}\|^2 \leq T \sum_{t=0}^T\left\|\mu_t\right\|^2 \leq \frac{TC(\widehat{K})}{\sigma_Q} .
$$
Taking square roots yields the third inequality:
$$
\|\mu_{\widehat{K}}\| \leq \sqrt{\frac{T C(\widehat{K})}{\underline{\sigma}_Q}}.
$$
\end{itemize}
\end{proof}

\subsection{Perturbation Analysis of State Statistics}
This section is to examine how the state statistics, namely the covariance $\Sigma_{\widehat{K}}$ and mean $\mu_{\widehat{K}}$, respond to changes in the policy $\widehat{K}$. We extend the operator-level bounds to derive a comprehensive perturbation analysis of $\Sigma_{\widehat{K}}$ and $\mu_{\widehat{K}}$, which quantifies the robustness of state evolution under policy updates.

First, let us define two linear operators on symmetric matrices. For $X \in \mathbb{R}^{(m+1) \times (m+1)}$ and $Y \in \mathbb{R}^{1 \times (m+1)}$ we set 
$$
\mathcal{F}_{\widehat{K}_t}(X)=(\widehat{A}_t-\widehat{B}_t \widehat{K}_t)X(\widehat{A}_t - \widehat{B}_t \widehat{K}_t)^{\top}, \quad \mathcal{F}_{\widehat{K}_t}^\mu(Y)=Y(\widehat{A}_t - \widehat{B}_t \widehat{K}_t)^{\top}.
$$
and 
$$
\mathcal{T}_{\widehat{K}}(X)=X+\sum_{t=0}^{T-1} \Pi_{i=0}^t(\widehat{A}_{i}-\widehat{B}_{i} \widehat{K}_{i}) X \Pi_{i=0}^t (\widehat{A}_i-\widehat{B}_i \widehat{K}_i)^{\top},
$$
$$
\mathcal{T}_{\widehat{K}}^\mu(Y)=Y+\sum_{t=0}^{T-1} Y \Pi_{i=0}^t (\widehat{A}_i-\widehat{B}_i \widehat{K}_i)^{\top}.
$$
If we write $\mathcal{G}_t=\mathcal{F}_{\widehat{K}_t} \circ \mathcal{F}_{\widehat{K}_{t-1}} \circ \cdots \circ \mathcal{F}_{\widehat{K}_0}$ and $\mathcal{G}_t^\mu=\mathcal{F}_{\widehat{K}_t}^\mu \circ \mathcal{F}_{\widehat{K}_{t-1}}^\mu \circ \cdots \circ \mathcal{F}_{\widehat{K}_0}^\mu$ then
\begin{align}
\mathcal{G}_t(X) & =\mathcal{F}_{\widehat{K}_t} \circ \mathcal{G}_{t-1}(X)=\Pi_{i=0}^t(\widehat{A}_{i}-\widehat{B}_{i} \widehat{K}_{i}) X \Pi_{i=0}^t (\widehat{A}_i-\widehat{B}_i \widehat{K}_i)^{\top}, \label{3.13}\\
\mathcal{G}_t^\mu(Y) & =\mathcal{F}_{\widehat{K}_t}^\mu \circ \mathcal{G}_{t-1}^\mu(Y)=Y \Pi_{i=0}^t (\widehat{A}_i-\widehat{B}_i \widehat{K}_i)^{\top}, \label{3.13u}\\
\mathcal{T}_{\widehat{K}}(X) & =X+\sum_{t=0}^{T-1} \mathcal{G}_t(X), \\
\mathcal{T}_{\widehat{K}}^\mu(Y) & =Y+\sum_{t=0}^{T-1} \mathcal{G}_t^\mu(Y).
\label{3.14}
\end{align}
We first show the relationship between the operator $\mathcal{T}_{\widehat{K}}$ and the quantity $\Sigma_{\widehat{K}}$, which pave the way for the relation between $\mathcal{T}_{\widehat{K}}^\mu$ and $\mu_{\widehat{K}}$.
\begin{proposition}
For $T \geq 2$, we have that
$$
\begin{aligned}
\Sigma_{\widehat{K}}&=\mathcal{T}_{\widehat{K}}\left(\Sigma_0\right)+\phi(\widehat{K}, \widehat{\phi}), \\
\mu_{\widehat{K}}&=\mathcal{T}_{\widehat{K}}^\mu\left(\mu_0\right)+\phi^\mu(\widehat{K}, \widehat{\phi}^\mu).
\end{aligned}
$$
Define $1) \phi(\widehat{K}, \widehat{\phi})=\sum_{t=1}^{T-1}\sum_{i=1}^{t} \Pi_{s=i}^{t}(\widehat{A}_{s}-\widehat{B}_{s}\widehat{K}_{s})\widehat{\phi}_{i-1}\Pi_{s=i}^{t}(\widehat{A}_s-\widehat{B}_s\widehat{K}_s)^{\top}+ \sum_{t=0}^{T-1}\widehat{\phi}_t$, with $\widehat{\phi} = 2\widehat{D}\bar{w}\mu(\widehat{A}-\widehat{B}\widehat{K})^{\top}+\widehat{D}\Sigma^w\widehat{D}^{\top}$ and $\Sigma_0=\mathbb{E}[\widehat{x}_0\widehat{x}_0^{\top}]$; $2)\phi(\widehat{K}, \widehat{\phi}^\mu)=\sum_{t=1}^{T-1}\sum_{i=1}^{t} \widehat{\phi}_{i-1}^\mu\Pi_{s=i}^{t}(\widehat{A}_s-\widehat{B}_s\widehat{K}_s)^{\top}+ \sum_{t=0}^{T-1}\widehat{\phi}_t^\mu$, with $\widehat{\phi}^\mu  = \bar{w}^{\top}\widehat{D}^{\top}$ and $\mu_0=\mathbb{E}[\widehat{x}_0^{\top}]$.
\label{prop3.9}
\end{proposition}
\begin{proof}%[\textbf{Proof of Proposition \ref{prop3.9}}]
We only need to prove for $\Sigma_{\widehat{K}}$, while the the proof of $\mu_{\widehat{K}}$ follows the same way. Recall that $\Sigma_t=\mathbb{E}\left[\widehat{x}_t\widehat{x}_t^{\top}\right]$. Note that
$$
\begin{aligned}
\Sigma_1 &=\mathbb{E}[\widehat{x}_1\widehat{x}_1^{\top}]\\
&=\mathbb{E}\left[\left((\widehat{A}_0-\widehat{B}_0\widehat{K}_0) \widehat{x}_0+\widehat{D}_0w_0\right)\left((\widehat{A}_0-\widehat{B}_0\widehat{K}_0) \widehat{x}_0+\widehat{D}_0w_0\right)^{\top}\right]\\
&=(\widehat{A}_0-\widehat{B}_0\widehat{K}_0)\Sigma_0 (\widehat{A}_0-\widehat{B}_0\widehat{K}_0)^{\top}+ 2\widehat{D}_0\bar{w}_0\mu_0(\widehat{A}_0-\widehat{B}_0\widehat{K}_0)^{\top}+\widehat{D}_0\Sigma_0^w\widehat{D}_0^{\top} \\
& = \mathcal{G}_0(\Sigma_0) + \widehat{\phi}_0.
\end{aligned}
$$
Now we first prove that
\begin{equation}
\begin{aligned}
\Sigma_t&=\mathcal{G}_{t-1}(\Sigma_0)+ \sum_{i=1}^{t-1} \Pi_{s=i}^{t-1}(\widehat{A}_{s}-\widehat{B}_{s}\widehat{K}_{s})\widehat{\phi}_{i-1}\Pi_{s=i}^{t-1}(\widehat{A}_s-\widehat{B}_s\widehat{K}_s)^{\top}+\widehat{\phi}_{t-1}, \\
\end{aligned}
\label{leC.5}
\end{equation}
for $t=2,3, \cdots, T$. When $t=2$,
$$
\begin{aligned}
\Sigma_2&=\mathbb{E}[\widehat{x}_2\widehat{x}_2^{\top}]\\
&=\mathbb{E}\left[\left((\widehat{A}_1-\widehat{B}_1\widehat{K}_1) \widehat{x}_1+\widehat{D}_1 w_1\right)\left((\widehat{A}_1-\widehat{B}_1\widehat{K}_1) \widehat{x}_1+\widehat{D}_1 w_1\right)^{\top}\right]\\
&=(\widehat{A}_1-\widehat{B}_1\widehat{K}_1)\Sigma_1(\widehat{A}_1-\widehat{B}_1\widehat{K}_1)^{\top}+ 2\widehat{D}_1\bar{w}_1\mu_1(\widehat{A}_1-\widehat{B}_1\widehat{K}_1)^{\top}+\widehat{D}_1\Sigma_1^w\widehat{D}_1^{\top}\\
& = \mathcal{G}_1(\Sigma_0)+ (\widehat{A}_1-\widehat{B}_1\widehat{K}_1)\widehat{\phi}_0 (\widehat{A}_1-\widehat{B}_1\widehat{K}_1)^{\top} +\widehat{\phi}_1,
\end{aligned}
$$
which satisfies (\ref{leC.5}). Assume (\ref{leC.5}) holds for $t \leq k$. Then for $t=k+1$,
$$
\begin{aligned}
\mathbb{E}[\widehat{x}_{t+1}\widehat{x}_{t+1}^{\top}]&=\mathbb{E}\left[\left((\widehat{A}_t-\widehat{B}_t\widehat{K}_t) \widehat{x}_t+\widehat{D}_tw_t\right)\left((\widehat{A}_t-\widehat{B}_t\widehat{K}_t) \widehat{x}_t+\widehat{D}_t w_t\right)^{\top}\right]\\
&=(\widehat{A}_t-\widehat{B}_t\widehat{K}_t)\Sigma_t (\widehat{A}_t-\widehat{B}_t\widehat{K}_t)^{\top}+ 2\widehat{D}_t\bar{w}_t\mu_t(\widehat{A}_t-\widehat{B}_t\widehat{K}_t)^{\top}+\widehat{D}_t\Sigma_t^w\widehat{D}_t^{\top} \\
& = \mathcal{G}_t(\Sigma_0) + \sum_{i=1}^{t} \Pi_{s=i}^{t}(\widehat{A}_{s}-\widehat{B}_{s}\widehat{K}_{s})\widehat{\phi}_{i-1} \Pi_{s=i}^{t}(\widehat{A}_s-\widehat{B}_s\widehat{K}_s)^{\top}+\widehat{\phi}_t, \\
\end{aligned}
$$
Therefore (\ref{leC.5}) holds, $\forall t=1,2, \cdots, T$. Finally,
$$
\begin{aligned}
\Sigma_{\widehat{K}}&=\sum_{t=0}^T \Sigma_t \\
&=\Sigma_0+\sum_{t=0}^{T-1} \mathcal{G}_t\left(\Sigma_0\right)+ \sum_{t=1}^{T-1}\sum_{i=1}^{t} \Pi_{s=i}^{t}(\widehat{A}_{s}-\widehat{B}_{s}\widehat{K}_{s})\widehat{\phi}_{i-1}\Pi_{s=i}^{t}(\widehat{A}_s-\widehat{B}_s\widehat{K}_s)^{\top}+ \sum_{t=0}^{T-1} \widehat{\phi}_t \\
&=\mathcal{T}_{\widehat{K}}\left(\Sigma_0\right)+\phi(\widehat{K}, \widehat{\phi}).
\end{aligned}
$$
\end{proof}
Let
\begin{equation}
\rho=\max \left\{\max _{0 \leq t \leq T-1}\|\widehat{A}_t -\widehat{B}_t \widehat{K}_t\|, \max _{0 \leq t \leq T-1}\|\widehat{A}_t-\widehat{B}_t \widehat{K}_t^{\prime}\|, 1+\xi\right\}   
\label{3.16}
\end{equation}
for some small constant $\xi>0$. Then we have the following result on perturbations of $\Sigma_{\widehat{K}}$ and $\mu_{\widehat{K}}$. Lemma \ref{le3.12} and \ref{le3.13} establish Lipschitz continuity of the associated operators $\mathcal{F}_{\widehat{K}_t}$ (and its mean counterpart $\mathcal{F}_{\widehat{K}_t}^\mu$) as well as $\mathcal{G}_t$ (and $\mathcal{G}_t^\mu$). 
\begin{lemma}
It holds that, $\forall t=0,1, \cdots, T-1$,
\begin{equation}
\begin{aligned}
\|\mathcal{F}_{\widehat{K}_t}-\mathcal{F}_{\widehat{K}_t^{\prime}}\| &\leq 2\|\widehat{A}_t-\widehat{B}_t \widehat{K}_t\|\|\widehat{B}_t\|\|\widehat{K}_t-\widehat{K}_t^{\prime}\|+{\|\widehat{B}_t\|}^2 {\|\widehat{K}_t-\widehat{K}_t^{\prime}\|}^2,\\
\|\mathcal{F}_{\widehat{K}_t}^\mu-\mathcal{F}_{\widehat{K}_t^{\prime}}^\mu\| &\leq \|\widehat{B}_t\|\|\widehat{K}_t-\widehat{K}_t^{\prime}\|.
\end{aligned}
\end{equation}
\label{le3.12}
\end{lemma}
\begin{proof}
Let $\Delta \widehat{K}_t=\widehat{K}_t-\widehat{K}_t^{\prime}$ and $\widehat{A}^{\mathrm{cl}}_t=\widehat{A}_t-\widehat{B}_t \widehat{K}_t$. Then
$$
\widehat{A}^{\prime \mathrm{cl}}_t=\widehat{A}_t-\widehat{B}_t \widehat{K}^{\prime}_t=\widehat{A}^{\mathrm{cl}}_t+\widehat{B}_t \Delta \widehat{K}_t
$$
Thus,
$$
\begin{aligned}
\mathcal{F}_{\widehat{K}^{\prime}}(X) & = (\widehat{A}^{\mathrm{cl}}_t+\widehat{B}_t \Delta \widehat{K}_t) X (\widehat{A}^{\mathrm{cl}}_t+\widehat{B}_t \Delta \widehat{K}_t)^{\top} \\
& =\widehat{A}^{\mathrm{cl}}_t X (\widehat{A}^{\mathrm{cl}})^{\top}_t + 2\widehat{A}^{\mathrm{cl}}_t X(\Delta \widehat{K}_t)^{\top} \widehat{B}^{\top}_t+ \widehat{B}_t \Delta \widehat{K}_t X (\Delta \widehat{K}_t)^{\top} \widehat{B}^{\top}_t.
\end{aligned}
$$
and
$$
\mathcal{F}_{\widehat{K}^{\prime}}^\mu(Y) = Y (\widehat{A}^{\mathrm{cl}}_t+\widehat{B}_t \Delta \widehat{K}_t)^{\top} =Y (\widehat{A}^{\mathrm{cl}})^{\top}_t + Y(\Delta \widehat{K}_t)^{\top} \widehat{B}^{\top}_t.
$$
The differences are
$$
\begin{aligned}
\mathcal{F}_{\widehat{K}}(X)-\mathcal{F}_{\widehat{K}^{\prime}}(X) &=-2\widehat{A}^{\mathrm{cl}}_t X (\Delta \widehat{K})^{\top}_t \widehat{B}^{\top}_t -\widehat{B}_t \Delta \widehat{K}_t X(\Delta \widehat{K}_t)^{\top} \widehat{B}^{\top}_t, \\
\mathcal{F}_{\widehat{K}}^\mu(Y)-\mathcal{F}_{\widehat{K}^{\prime}}^\mu(Y) &= - Y(\Delta \widehat{K}_t)^{\top} \widehat{B}^{\top}_t \\
\end{aligned}
$$
Take operator norms and use submultiplicativity $\|M N\| \leq\|M\|\|N\|$ :
$$
\begin{aligned}
\|\mathcal{F}_{\widehat{K}}(X)-\mathcal{F}_{\widehat{K}^{\prime}}(X)\| &\leq 2 \|\widehat{A}^{\mathrm{cl}}_t\|\|\widehat{B}_t\|\|\Delta \widehat{K}_t\|\|X\|+\|\widehat{B}_t\|^2\|\Delta \widehat{K}\|^2\|X\|, \\
\|\mathcal{F}_{\widehat{K}}^\mu(Y)-\mathcal{F}_{\widehat{K}^{\prime}}^\mu(Y)\| &\leq \|\widehat{B}_t\|\|\Delta \widehat{K}_t\|\|Y\|.
\end{aligned}
$$
Setting $\|X\|=1$ and $\|Y\|=1$ gives
$$
\begin{aligned}
\|\mathcal{F}_{\widehat{K}}-\mathcal{F}_{\widehat{K}^{\prime}}\| &\leq 2 \|\widehat{A}^{\mathrm{cl}}_t\|\|\widehat{B}_t\|\|\Delta \widehat{K}_t\|+\|\widehat{B}_t\|^2\|\Delta \widehat{K}_t\|^2, \\
\|\mathcal{F}_{\widehat{K}}^\mu-\mathcal{F}_{\widehat{K}^{\prime}}^\mu\| &\leq \|\widehat{B}_t\|\|\Delta \widehat{K}_t\|.
\end{aligned}
$$
i.e.,
$$
\begin{aligned}
\|\mathcal{F}_{\widehat{K}}-\mathcal{F}_{\widehat{K}^{\prime}}\| &\leq 2\|\widehat{A}_t-\widehat{B}_t \widehat{K}_t\|\|\widehat{B}_t\|\|\widehat{K}_t-\widehat{K}^{\prime}_t\|+\|\widehat{B}_t\|^2\|\widehat{K}_t-\widehat{K}^{\prime}_t\|^2,\\
\|\mathcal{F}_{\widehat{K}}^\mu-\mathcal{F}_{\widehat{K}^{\prime}}^\mu\| &\leq \|\widehat{B}_t\|\|\widehat{K}_t-\widehat{K}^{\prime}_t\|.
\end{aligned}
$$
\end{proof}
Recall the definition of $\mathcal{G}_t$ in (\ref{3.13}) associated with $K$, similarly let us define $\mathcal{G}_t^{\prime}=\mathcal{F}_{K_t^{\prime}} \circ \mathcal{F}_{K_{t-1}^{\prime}} \circ$ $\cdots \circ \mathcal{F}_{K_0^{\prime}}$ for policy $K^{\prime}$. Then we have the following perturbation analysis for $\mathcal{G}_t$.
\begin{lemma}\textup{(Perturbation Analysis for $\mathcal{G}_t$ and $\mathcal{G}_t^\mu$)} 
For any symmetric matrix $\Sigma \in \mathbb{R}^{2 \times 2}$, and matrix $\mu \in \mathbb{R}^{1 \times 2}$, we have that
$$
\begin{aligned}
\sum_{t=0}^{T-1}\left\|\left(\mathcal{G}_t-\mathcal{G}_t^{\prime}\right)(\Sigma)\right\| &\leq \frac{\rho^{2 T}-1}{\rho^2-1}\left(\sum_{t=0}^{T-1}\|\mathcal{F}_{\widehat{K}_t}-\mathcal{F}_{\widehat{K}_t^{\prime}}\|\right)\|\Sigma\|, \mbox{ and}\\
\sum_{t=0}^{T-1}\left\|\left(\mathcal{G}_t^\mu-(\mathcal{G}_t^{\mu})' \right)(\mu)\right\| &\leq \frac{\rho^{T}-1}{\rho-1}\left(\sum_{t=0}^{T-1}\|\mathcal{F}_{\widehat{K}_t}^\mu-\mathcal{F}_{\widehat{K}_t^{\prime}}^\mu\|\right)\|\mu\|,
\end{aligned}
$$
\label{le3.13}
\end{lemma}
\begin{proof}%[\textbf{Proof of Lemma \ref{le3.13}}]
By direct calculation,
\begin{equation}
\|\mathcal{G}_t\| \leq \rho^{2(t+1)}, \quad \|\mathcal{G}_t^{\prime}\| \leq \rho^{2(t+1)} \quad \mbox{and} \quad \|\mathcal{G}_t^\mu\| \leq \rho^{t+1}, \quad \|(\mathcal{G}_t^{\mu})'\| \leq \rho^{t+1}.
\label{C.6}
\end{equation}
Denote $\mathcal{F}_t=\mathcal{F}_{\widehat{K}_t}$ and $\mathcal{F}_t^{\prime}=\mathcal{F}_{\widehat{K}_t^{\prime}}$ to ease the exposition. Then for any matrix $\Sigma \in \mathbb{R}^{2 \times 2}$ and $t \geq 0$,
$$
\begin{aligned}
\left\|\left(\mathcal{G}_{t+1}^{\prime}-\mathcal{G}_{t+1}\right)(\Sigma)\right\| & =\left\|\mathcal{F}_{t+1}^{\prime} \circ \mathcal{G}_t^{\prime}(\Sigma)-\mathcal{F}_{t+1} \circ \mathcal{G}_t(\Sigma)\right\| \\
& =\left\|\mathcal{F}_{t+1}^{\prime} \circ \mathcal{G}_t^{\prime}(\Sigma)-\mathcal{F}_{t+1}^{\prime} \circ \mathcal{G}_t(\Sigma)+\mathcal{F}_{t+1}^{\prime} \circ \mathcal{G}_t(\Sigma)-\mathcal{F}_{t+1} \circ \mathcal{G}_t(\Sigma)\right\| \\
& \leq\left\|\mathcal{F}_{t+1}^{\prime} \circ \mathcal{G}_t^{\prime}(\Sigma)-\mathcal{F}_{t+1}^{\prime} \circ \mathcal{G}_t(\Sigma)\right\|+\left\|\mathcal{F}_{t+1}^{\prime} \circ \mathcal{G}_t(\Sigma)-\mathcal{F}_{t+1} \circ \mathcal{G}_t(\Sigma)\right\| \\
& =\left\|\mathcal{F}_{t+1}^{\prime} \circ\left(\mathcal{G}_t^{\prime}-\mathcal{G}_t\right)(\Sigma)\right\|+\left\|\left(\mathcal{F}_{t+1}^{\prime}-\mathcal{F}_{t+1}\right) \circ \mathcal{G}_t(\Sigma)\right\| \\
& \leq\left\|\mathcal{F}_{t+1}^{\prime}\right\|\left\|\left(\mathcal{G}_t^{\prime}-\mathcal{G}_t\right)(\Sigma)\right\|+\left\|\mathcal{G}_t\right\|\left\|\mathcal{F}_{t+1}^{\prime}-\mathcal{F}_{t+1}\right\|\left\|\Sigma\right\| \\
& \leq \rho^2\left\|\left(\mathcal{G}_t^{\prime}-\mathcal{G}_t\right)(\Sigma)\right\|+\rho^{2(t+1)}\left\|\mathcal{F}_{t+1}^{\prime}-\mathcal{F}_{t+1}\right\|\left\|\Sigma\right\|.
\end{aligned}
$$
We define $\mathcal{F}_t^\mu=\mathcal{F}_{\widehat{K}_t}^\mu$ and $(\mathcal{F}_t^{\mu})'=(\mathcal{F}_{\widehat{K}_t}^\mu)'$ and can get similar inequality result for $\mu$. Therefore,
\begin{align}
\left\|\left(\mathcal{G}_{t+1}^{\prime}-\mathcal{G}_{t+1}\right)(\Sigma)\right\| &\leq \rho^2\left\|\left(\mathcal{G}_t^{\prime}-\mathcal{G}_t\right)(\Sigma)\right\|+\rho^{2(t+1)}\left\|\mathcal{F}_{t+1}^{\prime}-\mathcal{F}_{t+1}\right\|\left\|\Sigma\right\|, \label{C.7} \\
\left\|\left((\mathcal{G}_{t+1}^{\mu})'-\mathcal{G}_{t+1}^\mu\right)(\mu)\right\| &\leq \rho \left\|\left((\mathcal{G}_{t+1}^{\mu})'-\mathcal{G}_{t+1}^\mu\right)(\mu)\right\|+\rho^{t+1}\left\|(\mathcal{F}_{t+1}^{\mu})'-\mathcal{F}_{t+1}^\mu\right\|\left\|\mu\right\|, \label{C.7-1} 
\end{align}
Summing (\ref{C.7}) and (\ref{C.7-1}) up for $t \in\{1,2, \cdots, T-2\}$ with $\left\|\mathcal{G}_0^{\prime}-\mathcal{G}_0\right\|=\left\|\mathcal{F}_0^{\prime}-\mathcal{F}_0\right\|$ and $\left\|(\mathcal{G}_0^{\mu})'-\mathcal{G}_0^\mu\right\|=\left\|(\mathcal{F}_0^{\mu})'-\mathcal{F}_0^\mu\right\|$, we have
$$
\begin{aligned}
\sum_{t=0}^{T-1}\left\|\left(\mathcal{G}_t-\mathcal{G}_t^{\prime}\right)(\Sigma)\right\| &\leq \frac{\rho^{2 T}-1}{\rho^2-1}\left(\sum_{t=0}^{T-1}\|\mathcal{F}_t-\mathcal{F}_t^{\prime}\|\right)\left\|\Sigma\right\|, \\
\sum_{t=0}^{T-1}\left\|\left(\mathcal{G}_t^\mu-(\mathcal{G}_t^{\mu})' \right)(\mu)\right\| &\leq \frac{\rho^{T}-1}{\rho-1}\left(\sum_{t=0}^{T-1}\|\mathcal{F}_{\widehat{K}_t}^\mu-\mathcal{F}_{\widehat{K}_t^{\prime}}^\mu\|\right)\|\mu\|.
\end{aligned}
$$
\end{proof}
The following perturbation analysis on $\mathcal{T}$ follows immediately from Lemma \ref{le3.13}.
\begin{corollary}
For any symmetric matrix $\Sigma \in \mathbb{R}^{(m+1) \times (m+1)}$ and matrix $\mu \in \mathbb{R}^{1\times (m+1)}$, we have
$$
\begin{aligned}
\|(\mathcal{T}_{\widehat{K}}-\mathcal{T}_{\widehat{K}^{\prime}})(\Sigma)\| &\leq \frac{\rho^{2 T}-1}{\rho^2-1}\left(\sum_{t=0}^{T-1}\|\mathcal{F}_{\widehat{K}_t}-\mathcal{F}_{\widehat{K}_t^{\prime}}\|\right)\left\|\Sigma\right\|, \\
\|(\mathcal{T}_{\widehat{K}}^\mu-\mathcal{T}_{\widehat{K}^{\prime}}^\mu)(\mu)\| &\leq \frac{\rho^{T}-1}{\rho-1}\left(\sum_{t=0}^{T-1}\|\mathcal{F}_{\widehat{K}_t}^\mu-\mathcal{F}_{\widehat{K}_t^{\prime}}^\mu\|\right)\left\|\mu\right\|.
\end{aligned}
$$
where $\rho$ is defined in (\ref{3.16}).
\label{co3.14}
\end{corollary}
Now we are ready to present the Lemma \ref{le3.10}.
\begin{lemma}{\textup{(Perturbation Analysis of $\Sigma_{\widehat{K}}$ and $\mu_{\widehat{K}}$)}}.
%Assume Assumption 2.1 holds. Then
We have
$$
\begin{aligned}
\|\Sigma_{\widehat{K}}-\Sigma_{\widehat{K}^{\prime}}\| & \leq\left\|\left(\mathcal{T}_{\widehat{K}}-\mathcal{T}_{\widehat{K}^{\prime}}\right)\left(\Sigma_0\right)\right\|+ \|\phi(\widehat{K}, \widehat{\phi}) -  \phi(\widehat{K}', \widehat{\phi}) \| \\
& \leq \frac{\rho^{2 T}-1}{\rho^2-1} \left(\frac{C(\widehat{K})}{\underline{\sigma}_{{Q}}}+  T(2\rho \|\widehat{D}\| \|\bar{w}\| \| \mu\|+ \|\widehat{D}\|^2 \|\Sigma^w\|)\right)\\
&\quad \cdot \left(2 \rho \|\widehat{B}\| \ |||\mathbf{{\widehat{K}}-{\widehat{K}^{\prime}}}||| + \|\widehat{B}\|^2 \ |||\mathbf{{\widehat{K}}-{\widehat{K}^{\prime}}}|||^2\right) + 2\|\widehat{D}\| \|\bar{w}\| \|\mu\| \ |||\mathbf{{\widehat{K}}-{\widehat{K}^{\prime}}}|||, \mbox{ and }\\
\|\mu_{\widehat{K}}-\mu_{\widehat{K}^{\prime}}\| & \leq\|(\mathcal{T}_{\widehat{K}}^\mu-\mathcal{\widehat{T}}_{K^{\prime}}^\mu)\left(\mu_0\right)\|+ \|\phi(\widehat{K}, \widehat{\phi}^\mu) -  \phi(\widehat{K}', \widehat{\phi}^\mu) \| \\
& \leq \frac{\rho^{T}-1}{\rho-1} \left(\sqrt{\frac{TC(\widehat{K})}{\underline{\sigma}_{{Q}}}}+  T\|\widehat{D}\| \|\bar{w}\|\right)\|\widehat{B}\| \ |||\mathbf{{\widehat{K}}-{\widehat{K}^{\prime}}}|||.
\end{aligned}
$$
\label{le3.10}
where $\|\widehat{B}\| = \max_t\{\|\widehat{B}_t\|\}$ and the same definition for $\|\widehat{D}\|, \|\bar{w}\|, \|\mu\|$ and $\|\Sigma^w\|$.
\end{lemma} 
\begin{proof}%[\textbf{Proof of Lemma \ref{le3.10}}]
Using Lemma \ref{le3.12},
$$
\begin{aligned}
\sum_{t=0}^{T-1}\|\mathcal{F}_{\widehat{K}_t}-\mathcal{F}_{\widehat{K}_t^{\prime}}\| & =\sum_{t=0}^{T-1}\left(2\|\widehat{A}_t-\widehat{B}_t \widehat{K}_t\|\|\widehat{B}_t\|\|\widehat{K}_t-\widehat{K}_t^{\prime}\|+\|\widehat{B}_t\|^2 \|\widehat{K}_t-\widehat{K}_t^{\prime}\|^2 \right) \\
& \leq 2 \rho  \|\widehat{B}\| \sum_{t=0}^{T-1} \|\widehat{K}_t- \widehat{K}_t^{\prime}\|+  \|\widehat{B}\|^2\sum_{t=0}^{T-1}  \|\widehat{K}_t-\widehat{K}_t^{\prime}\|^2, \mbox{ and } \\
\sum_{t=0}^{T-1}\|\mathcal{F}_{\widehat{K}_t}^\mu-\mathcal{F}_{\widehat{K}_t^{\prime}}^\mu\| & =\sum_{t=0}^{T-1}\left(\|\widehat{B}_t\|\|\widehat{K}_t-\widehat{K}_t^{\prime}\|\right) \leq \|\widehat{B}\| \sum_{t=0}^{T-1} \|\widehat{K}_t- \widehat{K}_t^{\prime}\|.
\end{aligned}
$$
In the same way as for the proof of Lemma \ref{le3.13}, for $t=1, \cdots, T-1$, we have
\begin{equation}
\begin{aligned}
&\ \quad \sum_{i=1}^t\left\| \Pi_{s=i}^{t}(\widehat{A}_{s}-\widehat{B}_{s}\widehat{K}_{s})\widehat{\phi}_{i-1}\Pi_{s=i}^{t}(\widehat{A}_s-\widehat{B}_s\widehat{K}_s)^{\top} - \Pi_{s=i}^{t}(\widehat{A}_{s}-\widehat{B}_{s}\widehat{K}'_{s})\widehat{\phi}_{i-1}\Pi_{s=i}^{t}(\widehat{A}_s-\widehat{B}_s\widehat{K}'_s)^{\top}\right\| \\
&\leq \frac{\rho^{2 T}-1}{\rho^2-1}\left(\sum_{i=0}^t\|\mathcal{F}_{\widehat{K}_i}-\mathcal{F}_{\widehat{K}_i^{\prime}}\|\|\widehat{\phi}_{i-1}\|\right) \\
& \leq \frac{\rho^{2 T}-1}{\rho^2-1}\left(\sum_{i=0}^t\|\mathcal{F}_{\widehat{K}_i}-\mathcal{F}_{\widehat{K}_i^{\prime}}\|\right)(2\rho \|\widehat{D}\| \|\bar{w}\| \| \mu\|+ \|\widehat{D}\|^2 \|\Sigma^w\|), \mbox{ and } \\
&\ \quad \sum_{i=1}^t\left\|\widehat{\phi}_{i-1}^\mu \Pi_{s=i}^{t}(\widehat{A}_s-\widehat{B}_s\widehat{K}_s)^{\top} - \widehat{\phi}_{i-1}^\mu\Pi_{s=i}^{t}(\widehat{A}_s-\widehat{B}_s\widehat{K}'_s)^{\top}\right\| \\
&\leq \frac{\rho^{T}-1}{\rho-1}\left(\sum_{i=0}^t\|\mathcal{F}_{\widehat{K}_i}-\mathcal{F}_{\widehat{K}_i^{\prime}}\|\|\widehat{\phi}_{i-1}^\mu\|\right) \\
& \leq \frac{\rho^{T}-1}{\rho-1}\left(\sum_{i=0}^t\|\mathcal{F}_{\widehat{K}_i}-\mathcal{F}_{\widehat{K}_i^{\prime}}\|\right)\|\widehat{D}\| \|\bar{w}\|, 
\end{aligned}
\label{pg-eq3.20}
\end{equation}
where we define $\widehat{\phi}_{i-1}= \mathbf{0}_{(m+1) \times (m+1)}$ and $\widehat{\phi}_{i-1}^\mu= \mathbf{0}_{1\times (m+1)}$ when $i=0$. By Proposition \ref{prop3.9}, Corollary \ref{co3.14}, (\ref{3.14}) and (\ref{pg-eq3.20}), we have
\begin{equation}
\begin{aligned}
\|\Sigma_{\widehat{K}}-\Sigma_{\widehat{K}^{\prime}}\| & \leq\left\|(\mathcal{T}_{\widehat{K}}-\mathcal{\widehat{T}}_{K^{\prime}})\left(\Sigma_0\right)\right\|+ \|\phi(\widehat{K}, \widehat{\phi}) -  \phi(\widehat{K}', \widehat{\phi}) \| \\
& \leq \frac{\rho^{2 T}-1}{\rho^2-1}\left(\sum_{t=0}^{T-1}\|\mathcal{F}_{\widehat{K}_t}-\mathcal{F}_{\widehat{K}_t^{\prime}}\| \right) \left(\|\Sigma_0\| + T(2\rho \|\widehat{D}\| \|\bar{w}\| \| \mu\|+ \|\widehat{D}\|^2 \|\Sigma^w\|)\right) \\
&\quad + \sum_{t=0}^{T-1}2\|\widehat{D}\| \|\bar{w}\| \|\mu\| \|\widehat{K}_t - \widehat{K}'_t\| \\
& \leq \frac{\rho^{2 T}-1}{\rho^2-1} \left(\frac{C(K)}{\underline{\sigma}_{{Q}}}+  T(2\rho \|\widehat{D}\| \|\bar{w}\| \| \mu\|+ \|\widehat{D}\|^2 \|\Sigma^w\|)\right)\\
&\quad \cdot \left(2 \rho \|\widehat{B}\| \ |||\mathbf{{\widehat{K}}-{\widehat{K}^{\prime}}}||| + \|\widehat{B}\|^2 \ |||\mathbf{{\widehat{K}}-{\widehat{K}^{\prime}}}|||^2\right) + 2\|\widehat{D}\| \|\bar{w}\| \|\mu\| \ |||\mathbf{{\widehat{K}}-{\widehat{K}^{\prime}}}|||, \\
\|\mu_{\widehat{K}}-\mu_{\widehat{K}^{\prime}}\| & \leq\|(\mathcal{T}_{\widehat{K}}^\mu-\mathcal{\widehat{T}}_{K^{\prime}}^\mu)\left(\mu_0\right)\|+ \|\phi(\widehat{K}, \widehat{\phi}^\mu) -  \phi(\widehat{K}', \widehat{\phi}^\mu) \| \\
& \leq \frac{\rho^{T}-1}{\rho-1}\left(\sum_{t=0}^{T-1}\|\mathcal{F}_{\widehat{K}_t}^\mu-\mathcal{F}_{\widehat{K}_t^{\prime}}^\mu\| \right) \left(\|\Sigma_0\| + T\|\|\widehat{D}\|\bar{w}\| \right) \\
& \leq \frac{\rho^{T}-1}{\rho-1} \left(\sqrt{\frac{TC(\widehat{K})}{\underline{\sigma}_{{Q}}}}+  T\|\widehat{D}\| \|\bar{w}\|\right)\|\widehat{B}\| \ |||\mathbf{{\widehat{K}}-{\widehat{K}^{\prime}}}|||.
\end{aligned}
\label{3.21}
\end{equation}
The last inequalities for both holds since $\|\Sigma_0\| \leq\|\Sigma_{\widehat{K}}\| \leq \frac{C(\widehat{K})}{\underline{\sigma}_{{Q}}}$ and $\|\mu_0\| \leq\|\mu_{\widehat{K}}\| \leq \sqrt{\frac{TC(\widehat{K})}{\underline{\sigma}_{{Q}}}}$ by Lemma \ref{le3.8}.
\end{proof}

\subsection{Convergence and Complexity Guarantees}
We can provide global convergence Theorem \ref{th:3.3} after two preliminary Lemmas about admissible step-size conditions and properties for the cumulation of cost sequence gradient and  policy $\widehat{K}$. 
\begin{lemma}
We have 
\begin{equation}
\widehat{K}_t^{\prime}=\widehat{K}_t-\eta \nabla_t C(\widehat{K}),
\label{3.22}
\end{equation}
where
\begin{equation}
\eta \leq \min \left\{\frac{(\rho^2-1)(\rho-1) \underline{\sigma}_{{Q}} \sqrt{\underline{\sigma}_{{Q}}} \underline{\sigma}_{{X}}}{2 T(Dn_1+Dn_2)\|\widehat{B}\| \max _t\{\|\nabla_t C(\widehat{K})\|\}},\  \frac{1}{2 C_1}\right\}, 
\label{3.23}
\end{equation}
with
\begin{equation}
\begin{aligned}
Dn_1 & = \left(C(\widehat{K})+ \underline{\sigma}_{{Q}}T(2\rho \|\widehat{D}\| \|\bar{w}\| \| \mu\|+ \|\widehat{D}\|^2 \|\Sigma^w\|)\right)(2\rho+1)(\rho^{2T}-1)(\rho-1)\sqrt{\underline{\sigma}_{{Q}}},\\
Dn_2 & = \left(\sqrt{TC(\widehat{K})} + \sqrt{\underline{\sigma}_{{Q}}}T \|\widehat{D}\|\|\bar{w}\|\right)(\rho^T-1)(\rho^2-1) \underline{\sigma}_{{Q}}, \\
C_1&=\left(\frac{C(\widehat{K})}{\underline{\sigma}_{{Q}}}+  T(2\rho \|\widehat{D}\| \|\bar{w}\| \| \mu\|+ \|\widehat{D}\|^2 \|\Sigma^w\|)\right)
\left(\frac{(2 \rho+1)\|\widehat{B}\|(\rho^{2 T}-1)}{\left(\rho^2-1\right) \underline{\sigma}_{{X}}} \sum_{t=0}^{T-1}\|\nabla_t C(\widehat{K})\|\right)\\
&\quad +\frac{2\|\widehat{D}\| \|\bar{w}\| \|\mu\|}{\underline{\sigma}_{{X}}} \sum_{t=0}^{T-1}\|\nabla_t C(\widehat{K})\| + \left(\sqrt{\frac{TC(\widehat{K})}{\underline{\sigma}_Q}}+T \|\widehat{D}\|\|\bar{w}\| \right)\left(\frac{\|\widehat{B}\|(\rho^T-1)}{(\rho-1)\underline{\sigma}_{{X}}}\sum_{t=0}^{T-1}\|\nabla_t C(\widehat{K})\|  \right)  \\
&\quad +\frac{2 C(\widehat{K})}{\underline{\sigma}_{{Q}}} \sum_{t=0}^{T-1}\|R_t+\widehat{B}_{t}^{\top} \widehat{P}_{t+1} \widehat{B}_t \|.  
\end{aligned}
\label{3.24}
\end{equation}
Then we have
$$
C(\widehat{K}^{\prime})-C(\widehat{K}^*) \leq\left(1-2 \eta \underline{\sigma}_{{R}} \frac{\underline{\sigma}_{{X}}^2}{\|\Sigma_{\widehat{K}^*}\|+1}\right)\left(C(\widehat{K})-C(\widehat{K}^*)\right).
$$
\label{le3.15}
\end{lemma}
\begin{proof}%[\textbf{Proof of Lemma \ref{le3.15}}]
Given (\ref{3.22}) and condition (\ref{3.23}), we have $\|\widehat{K}_t^{\prime}-\widehat{K}_t\|=\eta\|\nabla_t C(\widehat{K})\| \leq$ $\frac{\underline{\sigma}_{{Q}} \underline{\sigma}_{{X}}}{2 C(\widehat{K})\|\widehat{B}\|}$. Therefore,
\begin{equation}
\|\widehat{B}_t\|\|\widehat{K}_t^{\prime}-\widehat{K}_t\| \leq \frac{\underline{\sigma}_{{Q}} \underline{\sigma}_{{X}}}{2 C(\widehat{K})} \leq \frac{1}{2}.
\label{leC.8}
\end{equation}
The last inequality holds since $\underline{\sigma}_{X} \leq \frac{C(\widehat{K})}{\underline{\sigma}_{{Q}}}$ given by Lemma \ref{le3.8}. Therefore, by Lemma \ref{le3.12},
\begin{equation}
\sum_{t=0}^{T-1}\|\mathcal{F}_{K_t}-\mathcal{F}_{K_t^{\prime}}\| \leq(2 \rho+1)\|\widehat{B}\|\left(\sum_{t=0}^{T-1}\|\widehat{K}_t-\widehat{K}_t^{\prime}\|\right).
\label{C.9}
\end{equation}
Define $\widetilde{\Sigma}_t = \begin{bmatrix}\Sigma_t \\ \mu_t \end{bmatrix}$, $\widetilde{\Sigma}'_t = \begin{bmatrix}\Sigma'_t \\ \mu'_t \end{bmatrix}$ and $\widetilde{E}_t = \begin{bmatrix}E_t & -F_t \end{bmatrix}$. By Lemmas \ref{le3.5} and \ref{le3.7}, we have
\begin{equation}
\begin{aligned}
&\quad \ C(\widehat{K}^{\prime})-C(\widehat{K})\\
&= \sum_{t=0}^{T-1}\left[2\operatorname{Tr}\left(\widetilde{\Sigma}_t^{\prime}(\widehat{K}_t^{\prime}-\widehat{K}_t)^{\top} \widetilde{E}_t\right) + \operatorname{Tr}\left(\Sigma_t^{\prime}(\widehat{K}_t^{\prime}-\widehat{K}_t)^{\top}(R_t+\widehat{B}_{t}^{\top}\widehat{P}_{t+1} \widehat{B}_t)(\widehat{K}_t^{\prime}-\widehat{K}_t)\right) \right] \\
&= \sum_{t=0}^{T-1}\left[-4 \eta \operatorname{Tr} \left(\widetilde{\Sigma}_t^{\prime} \widetilde{\Sigma}_t^{\top} \widetilde{E}_t^{\top} \widetilde{E}_t\right) +4 \eta^2 \operatorname{Tr} \left(\Sigma_t^{\prime} \widetilde{\Sigma}_t^{\top} \widetilde{E}_t^{\top}(R_t+\widehat{B}_{t}^{\top} \widehat{P}_{t+1} \widehat{B}_t) \widetilde{E}_t \widetilde{\Sigma}_t\right) \right] \\
&=  \sum_{t=0}^{T-1}\left[-4 \eta \operatorname{Tr}\left((\widetilde{\Sigma}_t^{\prime}-\widetilde{\Sigma}_t+\widetilde{\Sigma}_t) \widetilde{\Sigma}_t^{\top} \widetilde{E}_t^{\top} \widetilde{E}_t\right) +4 \eta^2 \operatorname{Tr} \left(\Sigma_t^{\prime} \widetilde{\Sigma}_t^{\top} \widetilde{E}_t^{\top}(R_t+\widehat{B}_{t}^{\top} \widehat{P}_{t+1} \widehat{B}_t) \widetilde{E}_t \widetilde{\Sigma}_t\right) \right] \\
&\leq  \sum_{t=0}^{T-1}\Bigl[-4 \eta \operatorname{Tr} \left(\widetilde{\Sigma}_t^{\top} \widetilde{E}_t^{\top} \widetilde{E}_t \widetilde{\Sigma}_t\right) +4 \eta \operatorname{Tr}\left((\widetilde{\Sigma}_t^{\prime}-\widetilde{\Sigma}_t) \widetilde{\Sigma}_t^{\top} \widetilde{E}_t^{\top} \widetilde{E}_t \widetilde{\Sigma}_t \widetilde{\Sigma}_t^{\dagger}\right) \\
&\quad +4 \eta^2 \operatorname{Tr} \left(\Sigma_t^{\prime} \widetilde{\Sigma}_t^{\top} \widetilde{E}_t^{\top}(R_t+\widehat{B}_{t}^{\top} \widehat{P}_{t+1} \widehat{B}_t) \widetilde{E}_t \widetilde{\Sigma}_t\right)\Bigr] \\
&\leq  \sum_{t=0}^{T-1}\Biggl[-4 \eta \operatorname{Tr} \left(\widetilde{\Sigma}_t^{\top} \widetilde{E}_t^{\top} \widetilde{E}_t \widetilde{\Sigma}_t\right)  +4 \eta \frac{\|\widetilde{\Sigma}'_t- \widetilde{\Sigma}_t\|}{\sigma_{\min}({\Sigma}_t)} \operatorname{Tr}\left(\widetilde{\Sigma}_t \widetilde{E}_t^{\top} \widetilde{E}_t \widetilde{\Sigma}_t\right)  \\
&\quad +4 \eta^2 \|\Sigma_t^{\prime}(R_t+\widehat{B}_{t}^{\top} \widehat{P}_{t+1} \widehat{B}_t)\| \operatorname{Tr} \left(\widetilde{\Sigma}_t \widetilde{E}_t^{\top} \widetilde{E}_t \widetilde{\Sigma}_t\right)\Biggr] \\
&\leq  -\eta\left[1-\frac{\sum_{t=0}^{T-1}\|\widetilde{\Sigma}_t^{\prime}-\widetilde{\Sigma}_t\|}{\underline{\sigma}_X}-\eta\|\Sigma_{\widehat{K}'}\| \sum_{t=0}^{T-1}\|R_t+\widehat{B}_{t}^{\top} \widehat{P}_{t+1} \widehat{B}_t \|\right] \sum_{t=0}^{T-1}\left[\nabla_t C(\widehat{K})^{\top} \nabla_t C(\widehat{K})\right].
\end{aligned}
\end{equation}
By Lemma \ref{le3.6}, we have
\begin{equation}
\begin{aligned}
C(\widehat{K}^{\prime})-C(\widehat{K}) &\leq  -\eta\left[1-\frac{\sum_{t=0}^{T-1}\|\widetilde{\Sigma}_t^{\prime}-\widetilde{\Sigma}_t\|}{\underline{\sigma}_{{X}}}-\eta\|\Sigma_{\widehat{K}'}\| \sum_{t=0}^{T-1}\|R_t+\widehat{B}_{t}^{\top} \widehat{P}_{t+1} \widehat{B}_t \|\right]  \\
&\quad \cdot \left(\frac{4 {\underline{\sigma}_{{X}}^2} \underline{\sigma}_{{R}}}{\|\Sigma_{\widehat{K}^*}\|+1}\right)\left(C(\widehat{K})-C(\widehat{K}^*)\right)
\end{aligned}
\end{equation}
provided that
\begin{equation}
1-\frac{\sum_{t=0}^{T-1}\|\widetilde{\Sigma}_t^{\prime}-\widetilde{\Sigma}_t\|}{\underline{\sigma}_{{X}}}-\eta\|\Sigma_{\widehat{K}^{\prime}}\| \sum_{t=0}^{T-1}\|R_t+\widehat{B}_{t}^{\top} \widehat{P}_{t+1} \widehat{B}_t\|>0.
\label{leC.12}
\end{equation}
By (\ref{3.21}), (\ref{3.22}), and (\ref{leC.8}),
$$
\begin{aligned}
\sum_{t=0}^{T-1}\|\widetilde{\Sigma}_t^{\prime}-\widetilde{\Sigma}_t\| &\leq \sum_{t=0}^{T-1}\|\Sigma_t^{\prime}-\Sigma_t\| +  \sum_{t=0}^{T-1}\|\mu_t^{\prime}-\mu_t\| \\
&\leq \frac{\rho^{2 T}-1}{\rho^2-1} \left(\frac{C(\widehat{K})}{\underline{\sigma}_{{Q}}}+  T(2\rho \|\widehat{D}\| \|\bar{w}\| \| \mu\|+ \|\widehat{D}\|^2 \|\Sigma^w\|)\right)\\
&\quad \cdot \left(\eta(2 \rho+1)\|\widehat{B}\| \sum_{t=0}^{T-1}\|\nabla_t C(\widehat{K})\|\right) + 2\eta\|\widehat{D}\| \|\bar{w}\| \|\mu\| \sum_{t=0}^{T-1}\|\nabla_t C(\widehat{K})\| \\
&\quad + \frac{\rho^T-1}{\rho-1}\left(\sqrt{\frac{TC(\widehat{K})}{\underline{\sigma}_Q}}+T \|\widehat{D}\|\|\bar{w}\| \right)\left(\eta \|\widehat{B}\|\sum_{t=0}^{T-1}\|\nabla_t C(\widehat{K})\|\right).
\end{aligned}
$$
Given the step size condition in (\ref{3.23}), we have
\begin{equation}
\begin{aligned}
\eta(2 \rho+1)\|\widehat{B}\| \sum_{t=0}^{T-1}\|\nabla_t C(\widehat{K})\| &\leq \eta(2 \rho+1)\|\widehat{B}\|\left(T \cdot \max _t\left\{\|\nabla_t C(\widehat{K})\|\right\}\right) \\
&\leq \frac{\left(\rho^2-1\right) \underline{\sigma}_{{Q}} \underline{\sigma}_{{X}}}{2\left(\rho^{2 T}-1\right)Dn_1}.
\label{leC.13}
\end{aligned}
\end{equation}
Then, by Corollary \ref{co3.14}, (\ref{3.21}), (\ref{C.9}) and (\ref{leC.13}),
$$
\begin{aligned}
\frac{\|\Sigma_{\widehat{K}^{\prime}}-\Sigma_{\widehat{K}}\|}{\underline{\sigma}_{{X}}} & \leq \frac{\rho^{2 T}-1}{\rho^2-1}\left(\sum_{t=0}^{T-1}\|\mathcal{F}_{\widehat{K}_t}-\mathcal{F}_{\widehat{K}_t^{\prime}}\| \right) \frac{\|\Sigma_0\| + T(2\rho \|\widehat{D}\| \|\bar{w}\| \| \mu\|+ \|\widehat{D}\|^2 \|\Sigma^w\|)}{\underline{\sigma}_{{X}}} \\
&\quad + \sum_{t=0}^{T-1}\frac{2\|\widehat{D}\| \|\bar{w}\| \|\mu\| \|\widehat{K}_t - \widehat{K}'_t\|}{\underline{\sigma}_{{X}}} \\
& \leq \frac{\rho^{2 T}-1}{\rho^2-1}(2 \rho+1)\|\widehat{B}\|\left(\sum_{t=0}^{T-1} \eta \| \nabla_t C(K)\|\right) \frac{C(K)+\underline{\sigma}_{{Q}} T (2\rho \|\widehat{D}\| \|\bar{w}\| \| \mu\|+ \|\widehat{D}\|^2 \|\Sigma^w\|)}{\underline{\sigma}_{{Q}} \underline{\sigma}_{{X}}} \\
&\quad + \frac{2\underline{\sigma}_{{Q}}\|\widehat{D}\| \|\bar{w}\| \| \mu \|}{\underline{\sigma}_{{Q}} \underline{\sigma}_{{X}}}\left(\sum_{t=0}^{T-1} \eta \| \nabla_t C(K)\|\right) \\
&\leq \frac{1}{2},
\end{aligned}
$$
Therefore, the bound of $\|\Sigma_{\widehat{K}^{\prime}}\|$ in (\ref{leC.12}) is given by
\begin{equation}
\|\Sigma_{\widehat{K}^{\prime}}\| \leq \|\Sigma_{\widehat{K}^{\prime}}-\Sigma_{\widehat{K}}\|+\|\Sigma_{\widehat{K}}\| \leq \frac{1}{2} \underline{\sigma}_{{X}}+\frac{C(\widehat{K} )}{\underline{\sigma}_{{Q}}} \leq \frac{1}{2}\|\Sigma_{\widehat{K}^{\prime}}\|+\frac{C(\widehat{K})}{\underline{\sigma}_{{Q}}},
\label{C.14}
\end{equation}
which indicates that $\|\Sigma_{\widehat{K}^{\prime}}\| \leq \frac{2 C(\widehat{K})}{{\underline{\sigma}}_{{Q}}}$. Therefore, (\ref{leC.12}) gives
$$
\begin{aligned}
&\quad \  1-\frac{\sum_{t=0}^{T-1}\|\widetilde{\Sigma}_t^{\prime}-\widetilde{\Sigma}_t\|}{\underline{\sigma}_{{X}}}-\eta\|\Sigma_{\widehat{K}^{\prime}}\| \sum_{t=0}^{T-1}\|R_t+\widehat{B}_{t}^{\top} \widehat{P}_{t+1} \widehat{B}_t\| \\
& \geq 1-\frac{\left(\rho^{2 T}-1\right)}{\left(\rho^2-1\right) \underline{\sigma}_{{X}}} \left(\frac{C(\widehat{K})}{\underline{\sigma}_{{Q}}}+  T(2\rho \|\widehat{D}\| \|\bar{w}\| \| \mu\|+ \|\widehat{D}\|^2 \|\Sigma^w\|)\right) \left(\eta(2 \rho+1)\|\widehat{B}\| \sum_{t=0}^{T-1}\|\nabla_t C(\widehat{K})\|\right)\\
&\quad - \frac{2\eta\|\widehat{D}\| \|\bar{w}\| \|\mu\|}{\underline{\sigma}_{{X}}} \sum_{t=0}^{T-1}\|\nabla_t C(\widehat{K})\| - \frac{\rho^T-1}{(\rho-1)\underline{\sigma}_{{X}}}\left(\sqrt{\frac{TC(\widehat{K})}{\underline{\sigma}_Q}}+T \|\widehat{D}\|\|\bar{w}\| \right)\left(\eta \|\widehat{B}\|\sum_{t=0}^{T-1}\|\nabla_t C(\widehat{K})\|  \right) \\
& \quad-\eta \frac{2 C(\widehat{K})}{\underline{\sigma}_{{Q}}} \sum_{t=0}^{T-1}\|R_t+\widehat{B}_{t}^{\top} \widehat{P}_{t+1} \widehat{B}_t \| \\
&=1-C_1 \eta,
\end{aligned}
$$
where $C_1$ is defined in (\ref{3.24}). So if $\eta \leq \frac{1}{2 C_1}$, then
$$
1-\frac{\sum_{t=0}^{T-1}\|\widetilde{\Sigma}_t^{\prime}-\widetilde{\Sigma}_t\|}{\underline{\sigma}_{{X}}}-\eta\|\Sigma_{\widehat{K}^{\prime}}\| \sum_{t=0}^{T-1}\|R_t+\widehat{B}_{t}^{\top} \widehat{P}_{t+1} \widehat{B}_t \| \geq 1-C_1 \eta \geq \frac{1}{2}>0 .
$$
Hence, 
$$C(\widehat{K}^{\prime})-C(\widehat{K}) \leq -\frac{\eta}{2}\left(\frac{4 {{\underline{\sigma}}_{{X}}^2} {\underline{\sigma}}_{{R}}}{\|\Sigma_{\widehat{K}^*}\|+1}\right)\left(C(\widehat{K})-C(\widehat{K}^*)\right),
$$
and 
$$
\begin{aligned}
C(\widehat{K}^{\prime})-C(\widehat{K}^*)&=\left(C(\widehat{K}^{\prime})-C(\widehat{K})\right)+\left(C(\widehat{K})-C(\widehat{K}^*)\right)\\
& \leq\left(1-2 \eta \frac{\underline{\sigma}_{{X}}^2 \underline{\sigma}_{{R}}}{\|\Sigma_{\widehat{K}^*}\|+1}\right)\left(C(\widehat{K})-C(\widehat{K}^*)\right).
\end{aligned}
$$
Thus, this lemma has been proven.
\end{proof}
\begin{lemma} 
We have that
$$
\sum_{t=0}^{T-1}\|\nabla_t C(\widehat{K})\|^2 \leq 4\left(\frac{TC(\widehat{K})}{\underline{\sigma}_{{Q}}}+\left(\frac{C(\widehat{K})}{\underline{\sigma}_{{Q}}}\right)^2\right) \frac{\max _t\|R_t+\widehat{B}_{t}^{\top} \widehat{P}_{t+1} \widehat{B}_t\|}{\underline{\sigma}_{{X}}+1}\left(C(\widehat{K})-C(\widehat{K}^*)\right),
$$
and that:
$$
\sum_{t=0}^{T-1}\|\widehat{K}_t\| \leq \frac{1}{\underline{\sigma}_{{R}}}\left(\sqrt{T \cdot \frac{\max _t\|R_t+\widehat{B}_{t}^{\top} \widehat{P}_{t+1} \widehat{B}_t \|}{\underline{\sigma}_{{X}}+1}\left(C(\widehat{K})-C(\widehat{K}^*)\right)}  +\sum_{t=0}^{T-1}\|{\widehat{B}_t}^{\top} \widehat{P}_{t+1} \widehat{A}_t \|\right).
$$
\label{le3.16}
\end{lemma}
\begin{proof}%[\textbf{Proof of Lemma \ref{le3.16}}]
Using Lemma \ref{le3.8} we have 
$$
\begin{aligned}
\sum_{t=0}^{T-1}\|\nabla_t C(\widehat{K})\|^2 &\leq 4 \sum_{t=0}^{T-1} \operatorname{Tr}\left(\widetilde{\Sigma}_t^{\top} \widetilde{E}_t^{\top} \widetilde{E}_t \widetilde{\Sigma}_t\right) \\
&\leq 4 \sum_{t=0}^{T-1}\|\widetilde{\Sigma}_t\|^2 \operatorname{Tr}(\widetilde{E}_t^{\top} \widetilde{E}_t) \\
&\leq 4\left(\frac{TC(\widehat{K})}{\underline{\sigma}_{{Q}}}+\left(\frac{C(\widehat{K})}{\underline{\sigma}_{{Q}}}\right)^2\right) \sum_{t=0}^{T-1}\operatorname{Tr}(\widetilde{E}_t^{\top} \widetilde{E}_t).
\end{aligned}
$$
From Lemma \ref{le3.6} we have
\begin{equation}
\begin{aligned}
C(\widehat{K})-C(\widehat{K}^*) &\geq (\underline{\sigma}_{{X}}+1) \sum_{t=0}^{T-1} \frac{1}{\|R_t+\widehat{B}_{t}^{\top} \widehat{P}_{t+1} \widehat{B}_t\|} \operatorname{Tr}(\widetilde{E}_t^{\top} \widetilde{E}_t) \\
&\geq \frac{\underline{\sigma}_{{X}}+1}{\max _t\|R_t+\widehat{B}_{t}^{\top} \widehat{P}_{t+1} \widehat{B}_t \|} \sum_{t=0}^{T-1} \operatorname{Tr}(\widetilde{E}_t^{\top} \widetilde{E}_t),      
\end{aligned}
\label{3.25}
\end{equation}
and hence
$$
\sum_{t=0}^{T-1}\|\nabla_t C(\widehat{K})\|^2 \leq 4\left(\frac{TC(\widehat{K})}{\underline{\sigma}_{{Q}}}+\left(\frac{C(\widehat{K})}{\underline{\sigma}_{{Q}}}\right)^2\right) \frac{\max _t\|R_t+\widehat{B}_{t}^{\top} \widehat{P}_{t+1} \widehat{B}_t\|}{\underline{\sigma}_{{X}}+1}\left(C(\widehat{K})-C(\widehat{K}^*)\right).
$$
For the second claim, using Lemma \ref{le3.6} again,
$$
\begin{aligned}
\sum_{t=0}^{T-1}\|\widehat{K}_t\| &=\sum_{t=0}^{T-1}\|(R_t+\widehat{B}_{t}^{\top} \widehat{P}_{t+1} \widehat{B}_t)^{-1} \widehat{K}_t(R_t+\widehat{B}_{t}^{\top} \widehat{P}_{t+1} \widehat{B}_t)\| \\
& \leq \sum_{t=0}^{T-1} \frac{1}{\sigma_{\min }(R_t)}\|\widehat{K}_t(R_t+\widehat{B}_{t}^{\top} \widehat{P}_{t+1} \widehat{B}_t)\|\\
& \leq \sum_{t=0}^{T-1} \frac{1}{\sigma_{\min }(R_t)}\left(\|E_t\|+\|\widehat{B}_t^{\top} \widehat{P}_{t+1} \widehat{A}_t \|\right) \\
& \leq \sum_{t=0}^{T-1}\left(\frac{\sqrt{\operatorname{Tr}(E_t^{\top} E_t)+F_t^{\top}F_t}}{\sigma_{\min }\left(R_t\right)}+\frac{\|{\widehat{B}_t}^{\top} \widehat{P}_{t+1} \widehat{A}_t\|}{\sigma_{\min }\left(R_t\right)}\right) \\
& \leq \frac{1}{\underline{\sigma}_{{R}}}\left(\sqrt{T \cdot \sum_{t=0}^{T-1} \operatorname{Tr}\left(E_t^{\top} E_t\right)+F_t^{\top}F_t}+\sum_{t=0}^{T-1}\|{\widehat{B}_t}^{\top} \widehat{P}_{t+1} \widehat{A}_t\|\right) \\
& \leq \frac{1}{\underline{\sigma}_{{R}}}\left(\sqrt{T \cdot \frac{\max _t\|R_t+\widehat{B}_{t}^{\top} \widehat{P}_{t+1} \widehat{B}_t\|}{\underline{\sigma}_{{X}}+1}\left(C(\widehat{K)}-C(\widehat{K}^*)\right)}+\sum_{t=0}^{T-1}\|{\widehat{B}_t}^{\top} \widehat{P}_{t+1} \widehat{A}_t\| \right).
\end{aligned}
$$
The second inequality holds by the definition of $E_t$ in (\ref{Et}), the second last step uses the Cauchy-Schwarz inequality, and the last inequality holds by (\ref{3.25}).
\end{proof}

\begin{theorem}[Global Convergence of the Gradient-Based Policy Update]
Assume that $\underline{\sigma}_{{X}}>0$ and the initial cost $C(\widehat{K}^{0})$ is finite. Let $\beta \in (0,1)$ denote a prescribed confidence level, and let $\varepsilon_M(\beta)$ be the radius of the Wasserstein ball $\mathbb{B}_{\varepsilon_M(\beta)}(\widehat{\mathbb{P}}_M)$ determined by Theorem \ref{m-con}.
Then, there exists a constant stepsize $\eta \in \mathcal{H}\left(\frac{1}{C(\widehat{K}^0)+1}\right)$ such that, with probability at least $1-\beta$ over the sampling of $\widehat{\mathbb{P}}_M$, the gradient-based policy update rule (\ref{EGDM}) satisfies
$$
\mathbb{P}^M \left\{C(\widehat{K}^N)-C(\widehat{K}^*) \leq \varepsilon \right\} > 1-\beta
$$
whenever
$$
N\geq \frac{\|\Sigma_{\widehat{K}^{*}}\|+1}{2 \eta \underline{\sigma}_{{X}}^2 \underline{\sigma}_{{R}}}\log \frac{C(\widehat{K}^0)-C(\widehat{K}^{*})}{\varepsilon}.
$$
\label{th:3.3}
\end{theorem}
\begin{proof}%[\textbf{Proof of Theorem \ref{th:3.3}}.]
To establish the existence of a positive stepsize $\eta$ ensuring condition (\ref{3.23}), it suffices to demonstrate that the right-hand side of (\ref{3.23}) admits a strictly positive lower bound. From Lemma \ref{le3.16} and by applying the Cauchy–Schwarz inequality, we have
\begin{equation}
\begin{aligned}
&\quad \ \sum_{t=0}^{T-1}\|\nabla_t C(\widehat{K})\| \\
&\leq \sqrt{T \cdot \sum_{t=0}^{T-1}\|\nabla_t C(\widehat{K})\|^2} \\
&\leq \sqrt{ 4T\left(\frac{TC(\widehat{K})}{\underline{\sigma}_{{Q}}}+\left(\frac{C(\widehat{K})}{\underline{\sigma}_{{Q}}}\right)^2\right) \frac{\max _t\|R_t+\widehat{B}_{t}^{\top} \widehat{P}_{t+1} \widehat{B}_t\|}{\underline{\sigma}_{{X}}+1}\left(C(\widehat{K})-C(\widehat{K}^*)\right)}.     
\end{aligned}
\label{3.26}
\end{equation}
Observe that for positive scalars $a, b, c, d>0$, the inequality $d<a b+c$ implies $\frac{1}{d}>\frac{1}{(a+1)(b+1)(c+1)}$ and for $a>0$ and $n \in \mathbb{N}^{+}$, we have $\frac{1}{a^n+1}>\frac{1}{(a+1)^n}$. Consequently, by busing inequalities (\ref{3.24}), we find that $\frac{1}{C_1}$ admits a polynomial lower bound. Next, we establish that $\frac{1}{\rho}$ is also bounded below by a polynomial in the same system parameters. From Lemma \ref{le3.15}, under the condition
$\|\widehat{B}_t\|\|\widehat{K}_t^{\prime}-\widehat{K}_t\| \leq \frac{\underline{\sigma}_Q \underline{\sigma}_X}{4 C(\widehat{K})} \leq \frac{1}{2}$, we have
$$
\max _{0 \leq t \leq T-1}\|\widehat{A}_t-\widehat{B}_t \widehat{K}_t^{\prime}\| \leq \max _t\|\widehat{A}_t-\widehat{B}_t \widehat{K}_t\|+\frac{1}{2}
$$
Hence,
$$
\begin{aligned}
\rho & =\max \left\{\max _{0 \leq t \leq T-1}\|\widehat{A}_t - \widehat{B}_t \widehat{K}_t\|, \max _{0 \leq t \leq T-1}\|\widehat{A}_t - \widehat{B}_t \widehat{K}_t^{\prime}\|, 1+\xi\right\} \\
& \leq \max \left\{\max _{0 \leq t \leq T-1}\|\widehat{A}_t - \widehat{B}_t \widehat{K}_t\|+\frac{1}{2}, 1+\xi\right\}\\
& \leq \max \left\{\|{\widehat{A}_t}\|+\|\widehat{B}_t\| \sum_{t=0}^{T-1}\|\widehat{K}_t\|+\frac{1}{2}, 1+\xi\right\} .
\label{3.27}
\end{aligned}
$$
By Lemma \ref{le3.16}, $\sum_{t=0}^{T-1}\|\widehat{K}_t\|$ is bounded, and Lemma \ref{le3.8} provides $\|\widehat{P}_t\|\leq \frac{C(\widehat{K})}{\underline{\sigma}_X}$.
Therefore, $\rho$ is bounded above by a polynomial in $\|\widehat{A}_t\|$, $\|\widehat{B}_t\|$, $|||{R}|||$, $\frac{1}{\underline{\sigma}_{{X}}}$, $\frac{1}{\underline{\sigma}_{{R}}}$ and $C(\widehat{K})$, or a constant $1+\xi$. Therefore $\frac{1}{\rho}$ can also be bounded below by the corresponding polynomials. Hence, by selecting $\eta \in \mathcal{H}\left(\frac{1}{C(\widehat{K}^0)+1}\right)$ as an appropriate polynomial function, the condition (\ref{3.23}) is satisfied since each gradient step guarantees that $C(\widehat{K}^1)<C(\widehat{K}^0)$. Applying Lemma \ref{le3.15}, we obtain the descent inequality:
$$
C(\widehat{K}^1)-C(\widehat{K}^*) \leq\left(1-2 \eta  \frac{\underline{\sigma}_{{X}}^2\underline{\sigma}_{{R}}}{\|\Sigma_{\widehat{K}^*}\|+1}\right)\left(C(\widehat{K}^0)-C(\widehat{K}^*)\right),
$$
which ensures a strict reduction in cost after the first iteration. By induction, suppose $C(\widehat{K}^n) \leq C(\widehat{K}^0)$. The stepsize condition in (\ref{3.23}) continues to hold, and applying Lemma \ref{le3.15} again yields
$$
C(\widehat{K}^{n+1})-C(\widehat{K}^*) \leq\left(1-2 \eta \frac{\underline{\sigma}_{{X}}^2 \underline{\sigma}_{{R}}}{\|\Sigma_{\widehat{K}^*}\|+1}\right)\left(C(\widehat{K}^n)-C(\widehat{K}^*)\right) .
$$
Thus, the cost sequence ${C(\widehat{K}^{n})}$ decreases geometrically, and for any $\varepsilon>0$, it follows that
$$
C(\widehat{K}^{N})-C(\widehat{K}^*) \leq \varepsilon \mbox{ whenever } N \geq \frac{\|\Sigma_{\widehat{K}^*}\|+1}{2 \eta \underline{\sigma}_X^2 \underline{\sigma}_R} \log \frac{C(\widehat{K}^0)-C(\widehat{K}^*)}{\varepsilon} .
$$
By the uniform performance guarantee (\ref{perf-gua}), with probability at least $1-\beta$ the empirical/DRO objective dominates the true objective simultaneously for all policies, hence for every iterate $\widehat{K}^i$ and for $\widehat{K}^*$. On this high-probability event, the true suboptimality gap at any iterate is bounded by the corresponding empirical gap. Since the first part of the proof already establishes geometric decay of the empirical gap and yields the iteration bound $N$ that makes the empirical gap $\leq \varepsilon$, the same bound applies to the true gap on the same event. Therefore,
$$
\mathbb{P}^M\left\{C(\widehat{K}^N)-C(\widehat{K}^*) \leq \varepsilon\right\} \geq 1-\beta, 
$$
which completes the proof.
\end{proof}

\section{Numerical Examples}
We present two numerical studies to illustrate the performance and properties of the proposed algorithm. The first examines a mean–variance (MV) portfolio optimization problem, highlighting how control constraints influence the optimal policy and cost. The second studies a dynamic benchmark tracking problem with additive uncertainty under a distributionally robust optimization (DRO) framework, focusing exclusively on the unknown noise setting and providing a sensitivity analysis with respect to the parameters.

\subsection{Mean-Variance Problem} \label{MV-intro}
We consider a standard multi-period mean-variance portfolio problem over a finite horizon. An investor starts with initial wealth $x_0$ and at each period $t=0, \ldots, T-1$ allocates wealth between one risk-free asset and $n$ risky assets. The risk-free asset yields a deterministic gross return $r_t$, while the risky assets have random return vector $e_t=\left(e_t^1, \ldots, e_t^n\right)^{\top}$. We assume that $e_0, \ldots, e_{T-1}$ are independent and are characterized by their first two moments: mean $\mathbb{E}\left(e_t\right)$ and covariance matrix $\operatorname{Cov}\left(e_t\right)$, which is positive semidefinite.

Let $x_t$ denote the investor's wealth at the beginning of period $t$, and let $u_t=\left(u_t^1, \ldots, u_t^n\right)^{\top}$ be the amount invested in the $n$ risky assets. Writing $ P_t=e_t-r_t \mathbf{1}$ for the excess return vector, the wealth processes as
$$
x_{t+1}=r_t x_t+P_t u_t, \quad t=0, \ldots, T-1.
$$
Following the classical mean–variance framework of Markowitz and its multi-period extensions (see, e.g., \cite{markowitz1952portfolio}, \cite{cui2014optimal} and \cite{wu2018explicit}), we fix a target level $d$ for the expected terminal wealth and consider
\begin{subequations}    
\begin{align}
\mbox { (MV) :}\quad &\min_{u_t} 
\operatorname{Var}\left[x_T\right]+\mathbb{E}  \left[ \sum_{t=0}^{T-1} u_t^{\top} R u_t\right] \label{main_func} \\
&\mbox { s.t. }  \mathbb{E}\left[x_T\right]=d, \label{target_d}\\
& \quad \quad x_{t+1}=r_t x_t+{P}_t {u}_t, \quad t=0, \cdots, T-1, \label{eqeg.1b}
\end{align}
\label{MV_finite}
\end{subequations}
Here $R \succeq 0$ penalizes large risky positions, and the constraint $d \geq x_0 \prod_{t=0}^{T-1} r_t$ ensures that the target exceeds the terminal wealth from investing only in the risk-free asset.

The Lagrangian treatment of the expectation constraint and the subsequent change of variables that shifts the state around the target are standard and lead to an equivalent LQ control formulation. Specifically, by introducing a scalar Lagrange multiplier and redefining the state, problem (MV) can be rewritten as a finite horizon stochastic LQ problem of the form
\begin{equation}
\begin{aligned}
\mbox { (MV-$\lambda_2$) :}\quad &\min_{u_t} 
\mathbb{E}\left[y_T^2\right]+ \mathbb{E}\left[\sum_{t=0}^{T-1}  u_t^{\top} R u_t\right]\\
&\mbox { s.t. } y_{t+1}=r_t y_t+{P}_t {u}_t, \quad t=0, \cdots, T-1.
\end{aligned}
\label{MV-L2}
\end{equation}
parametrized by the Lagrange multiplier. We follow exactly the derivation developed in our previous work \cite{zhang2025data} and refer to that paper for the detailed algebra and proofs. 

\subsubsection{Parameter Settings}
This section applies the MV problem to illustrate its formulation as a LQ problem with a scalar state. This example pertains to the finite-time horizon scenario, wherein the investor aims to minimize their risk over a span of $T$ periods. As mentioned in the preliminary section \ref{MV-intro}, the MV problem can be stated according to the model (\ref{MV-L2}) as follows: 
\begin{equation}
\min_{u_t} \mathbb{E}_\pi\left[\sum_{t=0}^{T-1} u_t^{\top} R_t u_t  + x_T^2\Bigg| x_0=x\right],
\label{ch3-example-1}
\end{equation}
which is subject to
\begin{equation}
x_{t+1}= r_t x_t + P_t u_t. 
\label{eq4-eg.2}
\end{equation}
Suppose the investor requires the proportion of each risky asset to be between 10\% and 20\% of the total wealth throughout all periods. This constraint can be expressed as:
$$
0.1 |x_t| \leq u_t^i \leq 0.2 |x_t|, \quad \forall i = 1, \cdots, n,
$$
where $x_t$ is the total wealth at time $t$, and $u_i$ represents the amount invested in risky asset $i$ at time $t$. The problem (MV) is evidently a specific instance of the problem presented in section \ref{ch3-setup} by setting $x_t \in \mathbb{R}$, $u_t \in \mathbb{R}^n$, $Q_t \in \mathbb{R}$, $R_t \in \mathbb{R}^{n \times n}$, $A_t=r_t$, $B_t= P_t$, $Q_t=0$ and $D_t = 0$ for $t=0, \cdots, T-1$ and  $Q_T=1$. For this state-dependent upper and lower bounds constraint case, we have $H_t=\begin{bmatrix}\mathbf{I}_n \\ -\mathbf{I}_n\end{bmatrix}$ and $d_t=\begin{bmatrix}\overline{d}_t\\ -\underline{d}_t\end{bmatrix}$ with $\overline{d}_t = \begin{bmatrix} 0.2 \\ 0.2 \end{bmatrix}$ and $\underline{d}_t = \begin{bmatrix} -0.1 \\ -0.1 \end{bmatrix}$. To ensure that our selected parameters meet the criteria for our LQ framework, we randomly choose the following parameters: $T = 5$ and $n = 3$. For the state-transition matrices $A_t$, we establish the risk-free rate $6\%$, resulting in $A_t = 1.06$. The expcted excess returns vector $B_t$ for the risky assets is expressed as $B_t = \begin{bmatrix} 14.8\% & 16\% & 9.8\% \end{bmatrix}$. The matrix $R_t$ is chosen randomly at each time interval. The initial wealth $x_0$ follows normal distribution, $x_0 \sim \mathcal{N}(10, 5 \times 10^{-2})$, with the currency expressed in thousand dollars. All entries in the initial policy matrix $K^0 \in \mathbb{R}^{n}$ are $0.08$.

\subsubsection{Results and Analysis}
We evaluate the performance of the proposed policy gradient algorithm under both known and unknown parameter settings. Performance is reported using the normalized error metric $\frac{C-C^*}{C}$, where $C$ denotes the cost associated with a given policy $K$, and $C^*$ represents the optimal cost corresponding to the benchmark policy $K^*$ defined in (\ref{real-k-star}).

\begin{figure}[h]
\centering
\subfigure{\includegraphics[width=0.48\textwidth]{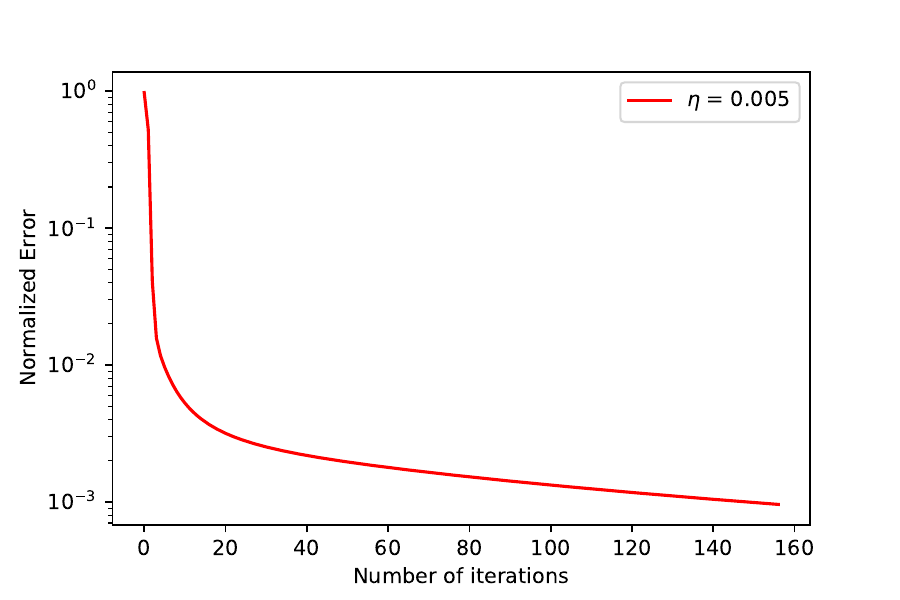}
\vspace{0.2mm}
\includegraphics[width=0.48\textwidth]{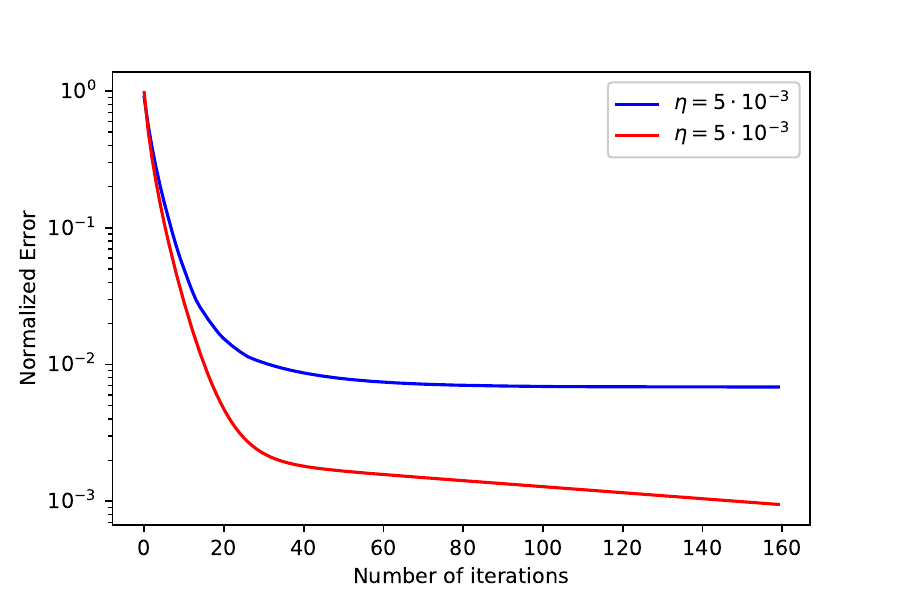}}
\caption{Performance of the policy gradient algorithm}
\label{figure1}
\end{figure}
Figure \ref{figure1} summarizes two aspects: convergence and the impact of the control constraint. The left shows fast convergence—the normalized error drops below $10^{-3}$ in roughly 160 iterations. The right compares constrained (blue) and unconstrained (red) updates under the same backtest: the constrained variant converges more slowly and settles at a slightly higher error level, consistent with a practical risk limit that tempers leverage and tail risk at a modest performance cost.
\begin{figure}[h]
\subfigure{\includegraphics[width=0.33\textwidth]{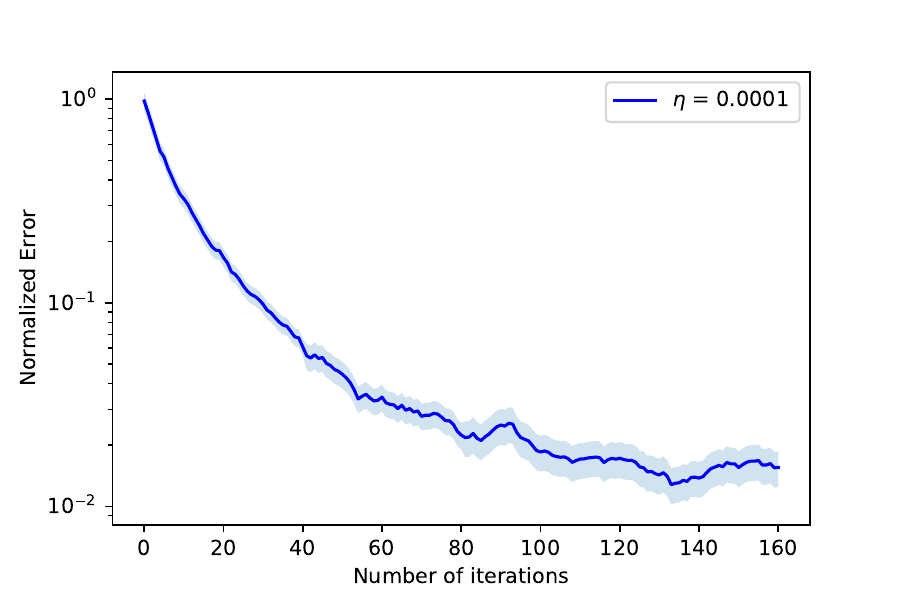}
\hspace{0.1mm}
\includegraphics[width=0.33\textwidth]{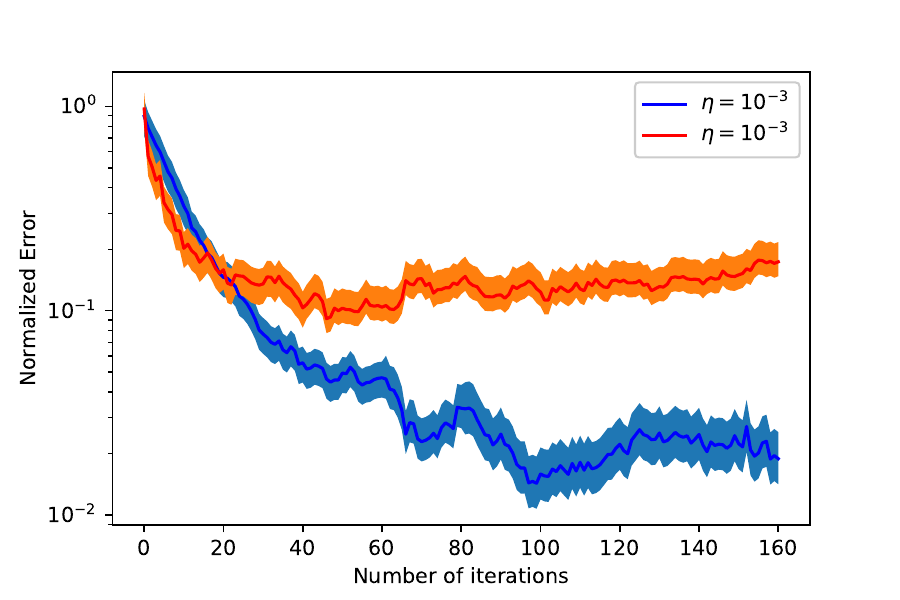}
\hspace{0.1mm}
\includegraphics[width=0.33\textwidth]{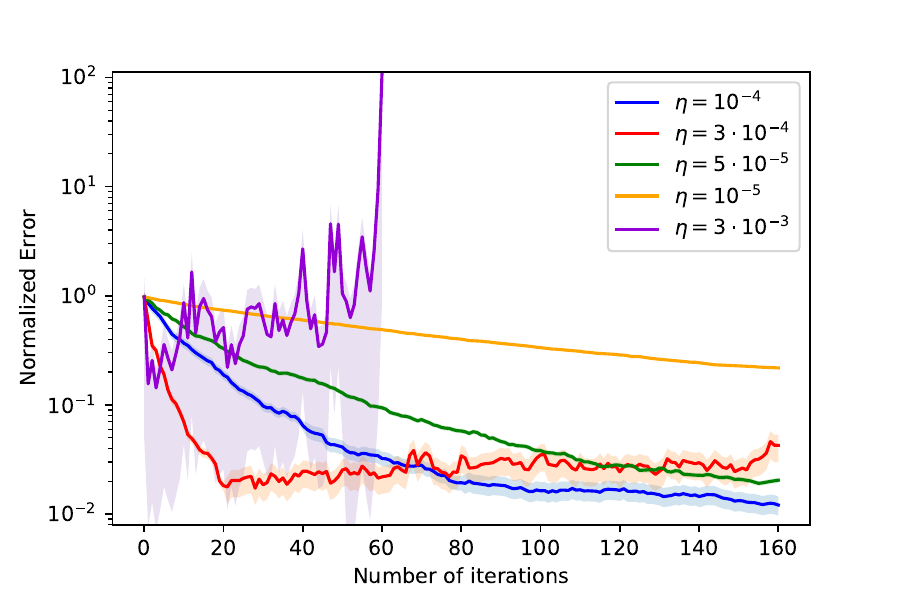}}
\caption{Comparison for constrained and non-constrained}
\label{figure2}
\end{figure}

When system parameters are unknown, Figure \ref{figure2} collects the corresponding results: convergence, the constrained (orange) vs. unconstrained (blue) comparison, and the step-size study. Convergence is noticeably slower because the algorithm must jointly estimate dynamics and optimize control; after a similar number of iterations, the error is around $10^{-2}$ (left). The constrained–unconstrained gap widens relative to the known-parameter case due to added estimation uncertainty (middle). The step-size analysis (right) shows that $\eta$ in the range $[10^{-5}, 3\times10^{-3}]$ materially affects learning: very small $\eta$ yields steady but slow progress with low fluctuation, while overly large $\eta$ can cause divergence.

\subsection{Dynamic Benchmark Tracking}
You manage a small basket of sector ETFs and want to track a benchmark (e.g., SPY) over a short horizon while minimizing (i) deviation from the benchmark and (ii) trading effort/cost. In practice, the “ideal” sector exposures that best replicate the benchmark drift over time due to composition/sector-mix changes. You trade to keep up with those drifting targets. This  benchmark tracking problem is well known to admit stochastic LQ ormulations; see, e.g., Yao \& Zhang \cite{yao2006tracking} about stochastic LQ tracking and later dynamic trading papers that cast portfolio rebalancing with quadratic costs in LQ form \cite{moallemi2017dynamic, garleanu2013dynamic}. Define the exposure deviation (state), trades (control) and noise term as follow:
\begin{itemize}
\item State (dimension = number of sector ETFs): $x_t \in \mathbb{R}^m$ is the exposure deviation at time $t$, which is actual sector holdings (in shares or exposure units) minus the target exposures that best track the benchmark at $t$.
\item Control: $u_t \in \mathbb{R}$ is the trades volumn changes in sector holdings. Positive means buy. 
\item The target exposure drift as exogenous noise: let $a_t^* \in \mathbb{R}^m$ be time-varying target exposures that best explain benchmark returns at time $t$, estimated from data. Its one-step change $w_t = a_{t+1}^*-a_t^*$ is an unpredictable, zero-mean additive noise to the exposure deviation dynamics. 
\end{itemize}
The dynamics are
$$
x_{t+1}=x_t-u_t+w_t,
$$
where $A_t=\mathbf{I}_m$, $B_t=-\mathbf{I}_m$, $D_t=\mathbf{I}_m$. We use a quadratic tracking objective (\ref{w-penalty}) with
$$
Q_t=\lambda_{\text{step }} \sigma_{\text{step }}^2 \mathbf{I}_m, \quad R_t=\operatorname{diag}\left(c_1, \ldots, c_m\right), \quad Q_T=\bar{q} \mathbf{I}_m(\bar{q} > 0) .
$$
Here $\lambda_{\text{step }} >0$ scales tracking risk via per-step variance $\sigma_{\text {step }}^2 ; R_t$ collects per ETF trading cost weights $c_j>0$; a large $\bar{q}$ enforces terminal closeness to the benchmark. 

With full state and no hard boxes, the classical LQ feedback is from the Riccati recursion. To hedge misspecified unknown or light-tailed $w_t$, we also use a   Wasserstein penalized DRO-LQ (\ref{w-penalty}), for which linear policies remain optimal and the backward recursion admits a simple modification via a positive-definite factor $\Delta$.

\subsubsection{Data Processing}
We use publicly available daily data from Yahoo Finance for the benchmark (e.g., SPY) and a set of sector ETFs (e.g., XLB, XLE, XLF, XLI, XLK, XLP, XLU, XLV, XLY, XLRE, XLC) from 2015/01/01 to 2025/09/30. See Figure \ref{norm-price}, which plots normalized adjusted close series for SPY and selected sectors over 2015–2025, illustrating broad co-movement with intermittent sector rotations for context only.
\begin{figure}[h] 
\centering
\includegraphics[width=0.8\textwidth]{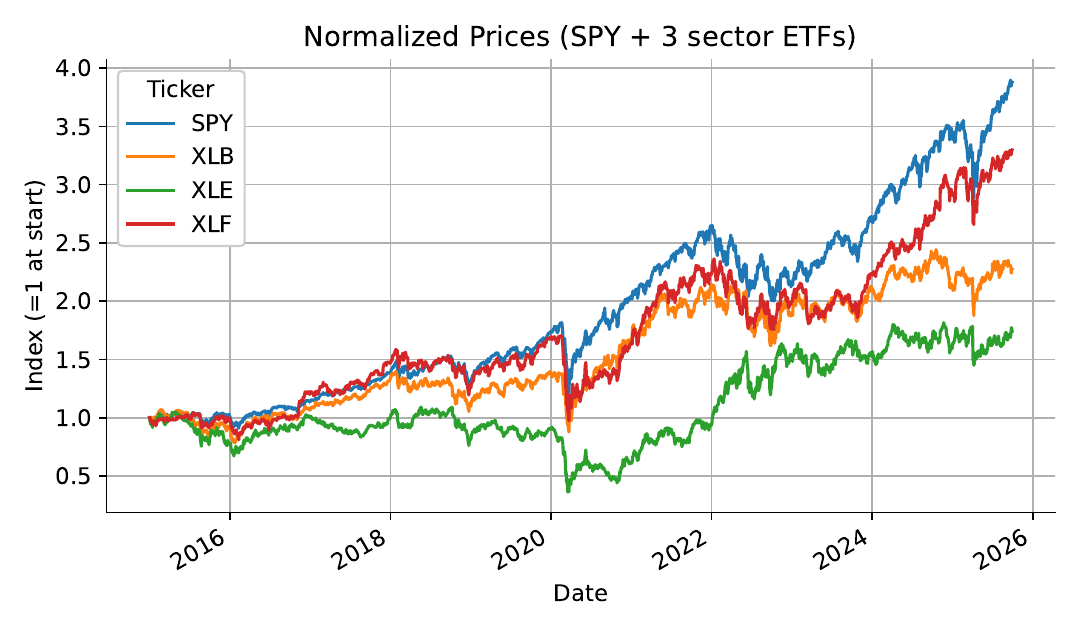} 
\caption{Normalized prices for ETFs and benchmark}
\label{norm-price}
\end{figure}
Series are aligned on U.S. trading days; rows with a missing benchmark price or any missing sector price are dropped. Let $P_t^{\text {bench }}$ denote the benchmark adjusted close and $P_t^{(j)}$ the adjusted close of sector $j \in\{1, \cdots, m\}$. Returns are computed as $\log$ differences $r_t^{\text {bench }}=\log P_t^{\text {bench }}-\log P_{t-1}^{\text {bench }}$ and $r_t^{(j)}=\log P_t^{(j)}- \log P_{t-1}^{(j)}$. To stabilize later volatility estimates, we optionally trim each return series at the $1^{\text{st }}$/$99^{\text{th }}$ percentiles within the calibration sample only. The Table \ref{tab:logret_latest_week} shows the log-return output for benchmark SPY other three sector ETFs. 

Time-varying target exposures $a_t^* \in \mathbb{R}^m$ are estimated by rolling ordinary least squares over a fixed lookback length $L_{\text{ols }}$ = 60 days. For each $t>L_{\text{ols }}$, we solve
$$
a_t^* \in \arg \min _{a \in \mathbb{R}^m} \sum_{\tau=t-L_{\mathrm{ols}}+1}^t\left(r_\tau^{\mathrm{bench}}-r_\tau^{\mathrm{sec} \top} a\right)^2, \quad r_\tau^{\mathrm{sec}}=\left(r_\tau^{(1)}, \ldots, r_\tau^{(m)}\right)^{\top},
$$
without intercept so that $a_t^*$ captures pure sector exposures (See Figure \ref{figure3}, left). If the normal matrix is ill-conditioned, a small ridge penalty $\alpha\|a\|_2^2$ (with $\alpha>0$ ) can be added; by default we take $\alpha=0$. Long only or budget constraints are optional post projections and are not required for the LQ formulation.
\begin{table}[h]
\centering
\begin{tabular}{|l|r|r|r|r|}
\hline
Date & SPY & XLB & XLC & XLE \\
\hline
2025-09-23 & -0.005458 & -0.003214 & -0.001435 &  0.017077 \\
2025-09-24 & -0.003187 & -0.012287 & -0.008311 &  0.012923 \\
2025-09-25 & -0.004624 & -0.013122 & -0.003412 &  0.008926 \\
2025-09-26 &  0.005713 &  0.011548 &  0.009694 &  0.009173 \\
2025-09-29 &  0.002806 &  0.003820 &  0.003211 & -0.018542 \\
\hline
\end{tabular}
\caption{Daily log-returns for SPY and three ETFs over the latest five trading days}
\label{tab:logret_latest_week}
\end{table}

The noise feeding the state equation is then taken directly from the data as the one-step drift in targets, $w_t=a_{t+1}^*-a_t^*$, which we treat as zero-mean and independent of the initial state(See Figure \ref{figure3}, right). The tracking state $x_t$ is the deviation between actual sector exposures and $a_t^*$; unless otherwise stated we initialize $x_0=\mathbf{0}_{m \times 1}$ at the start of the evaluation sample. 
\begin{figure}[h]
\centering
\subfigure{\includegraphics[width=0.48\textwidth]{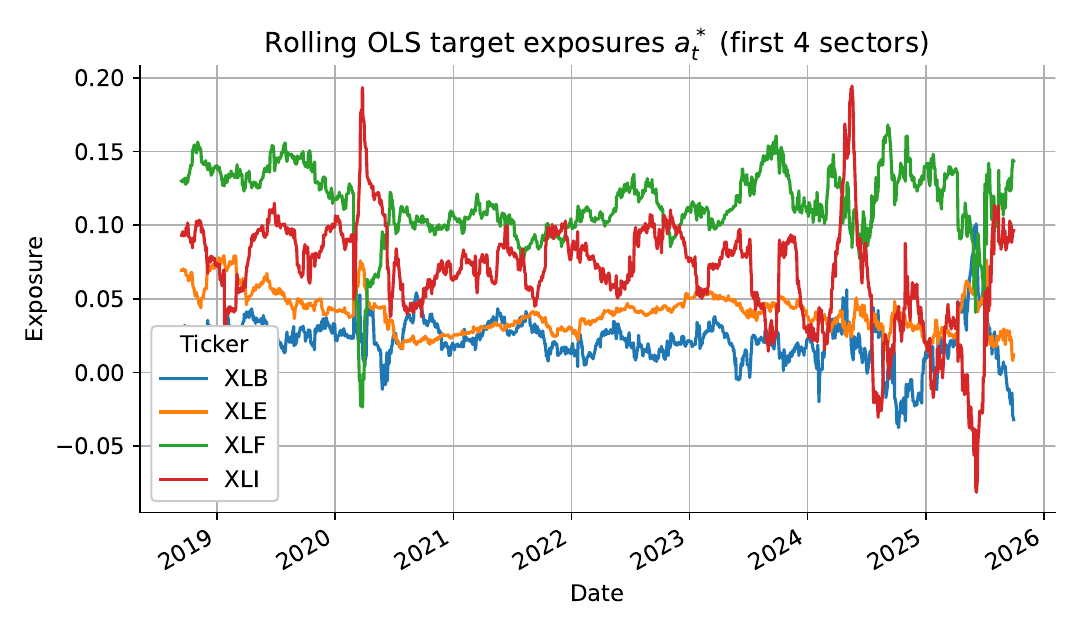}}
\hfill
\subfigure{\includegraphics[width=0.48\textwidth]{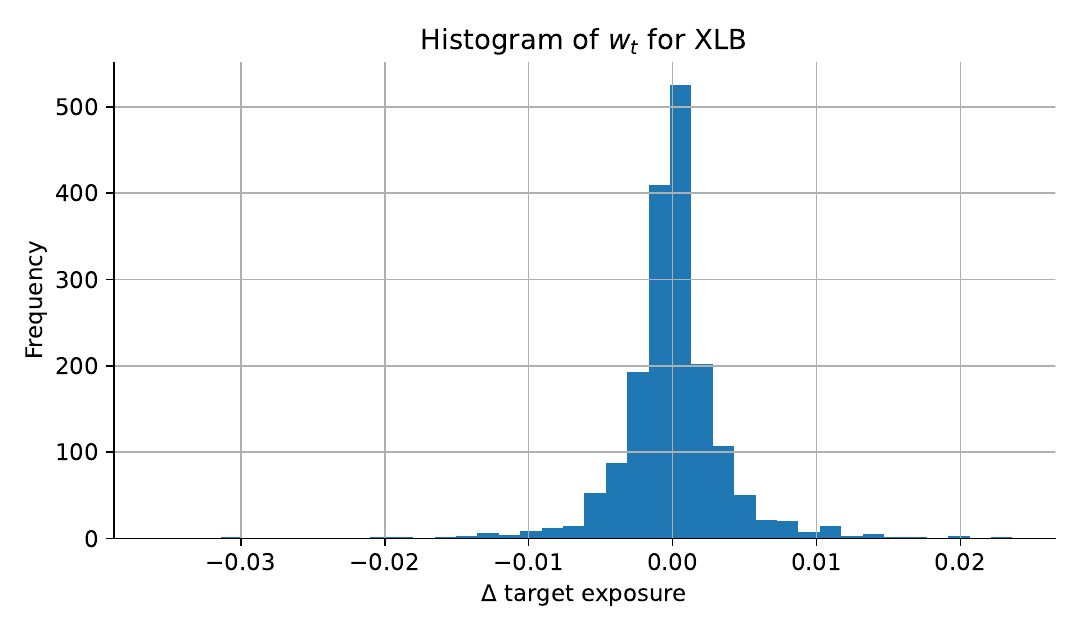}}
\caption{Empirical features for the tracking model}
\label{figure3}
\end{figure}
Risk scaling uses the benchmark's realized volatility. Let $\widehat{\sigma}_{\text {daily }}$ be the rolling standard deviation of $r_t^{\text {bench }}$ computed over $L_{\text {vol }}$ trading days, which is the same length as the OLS window. For a horizon of $T$ decision steps, the per step scale is
$$
\sigma_{\text {step }}=\widehat{\sigma}_{\text {daily }} / \sqrt{T},
$$
where $T = 120$. We set $Q_t=\lambda_{\text{step }}  \sigma_{\text {step }}^2 \mathbf{I}_m$ for all $t$ with $\lambda_{\text{step }}= 10$ and $\bar{q} = 10^6$ for $Q_T$. trading cost weights are diagonal, $R_t=\operatorname{diag}\left(c_1, \ldots, c_m\right)$, with each $c_j$ calibrated from observable liquidity via
$$
c_j=\kappa_{\mathrm{tc}} \frac{(S_0^{(j)})^2}{\mathrm{ADV}_j},
$$
where $S_0^{(j)}$ is the most recent adjusted close of sector $j$, $\mathrm{ADV}_j$ is the $L_{\text {adv }} = 60$ days moving average of its share volume, and $\kappa_{\mathrm{tc}}>0$ sets the overall trading effort scale (typical order $10^{-6}-10^{-5}$ ). This specification is dimensionally consistent when $x_t$ and $u_t$ are measured in shares; if exposures are in dollars, the same form applies after the natural rescaling by price \cite{benidis2018optimization}.

To prevent look ahead, $\widehat{\sigma}_{\text {daily }}, a_t^*$, and $\mathrm{ADV}_j$ are computed using only information available up to time $t$. Fixing $Q_t$, $R_t$, $\lambda_{\text{step }}$, $\kappa_{\mathrm{tc}}$, and $\bar{q}$, we calibrate on a training period, then evaluate on a hold-out period using the realized noise path $\{w_t\} ; Q_t$ and $R_t$ are kept at their last calibrated values. Dates with missing prices, hence undefined $w_t$, are omitted. The full pipeline is deterministic and thus fully reproducible.

\subsubsection{Evaluation Results}

\begin{figure}[h]
\centering 
\subfigure{\includegraphics[width=0.48\textwidth]{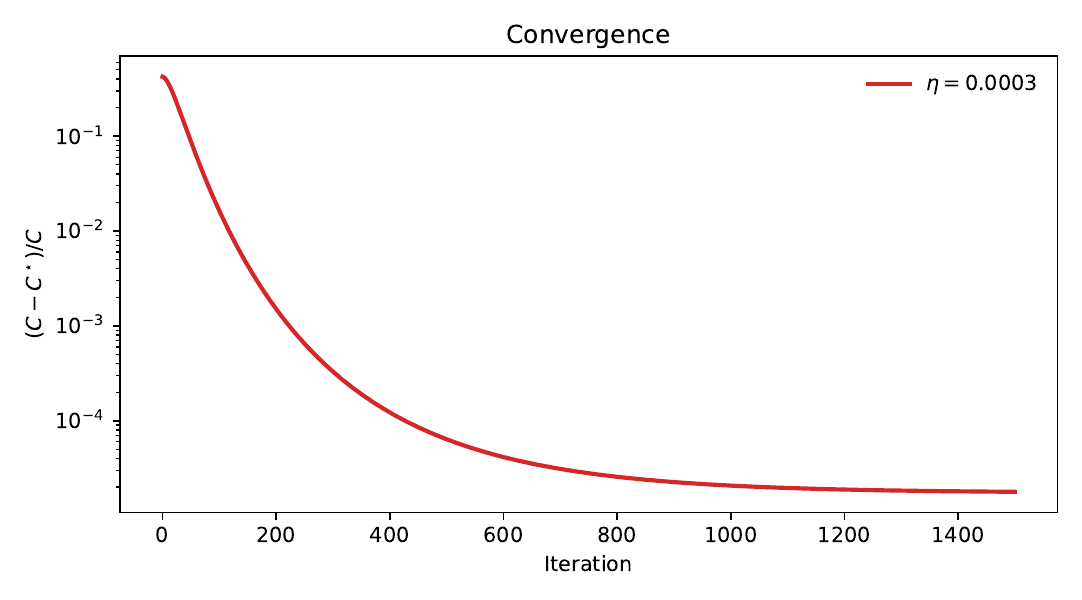}} 
\hfill 
\subfigure{\includegraphics[width=0.48\textwidth]{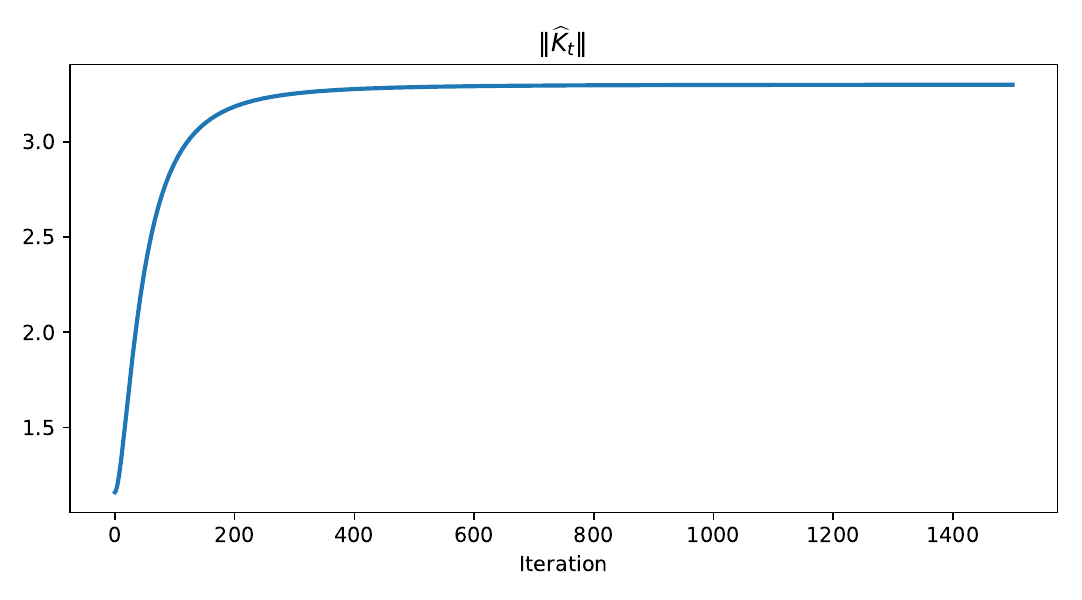}} 
\caption{DRO-LQ: policy gradient convergence and gain stabilization.} 
\label{figure4} 
\end{figure}
Two complementary views of the same learning run for the tracker in Figure \ref{figure4}. The left (normalized error) shows a smooth, monotone descent over roughly three orders of magnitude from about $10^{0}$ down to the low $10^{-4}$ range before flattening into a stable plateau, which indicates the policy gradient has essentially reached the Riccati benchmark and further gains are marginal. The right one (gain size) starts with $\|\widehat{K}_t\|$ from initial $\|\widehat{K}_0\|$, then contracts steadily to a tight band a little above 3 where it remains flat; this stabilization line-up in time with the error plateau, confirming the controller has settled into a steady, well-conditioned feedback policy.

We plot the percentage difference versus the baseline because the raw convergence curves sit on very different scales, making tiny gaps visually indistinguishable. Normalizing to a relative metric makes the series dimensionless, isolating the marginal effect of each parameter from mere scale shifts. We then apply a running mean to the difference series to suppress high-frequency noise (from sampling error and line search) while preserving the underlying trend.\\

\begin{figure}[h]
\centering
% -------- Row 1 --------
\begin{minipage}[b]{0.48\textwidth}
\centering
\includegraphics[width=\linewidth]{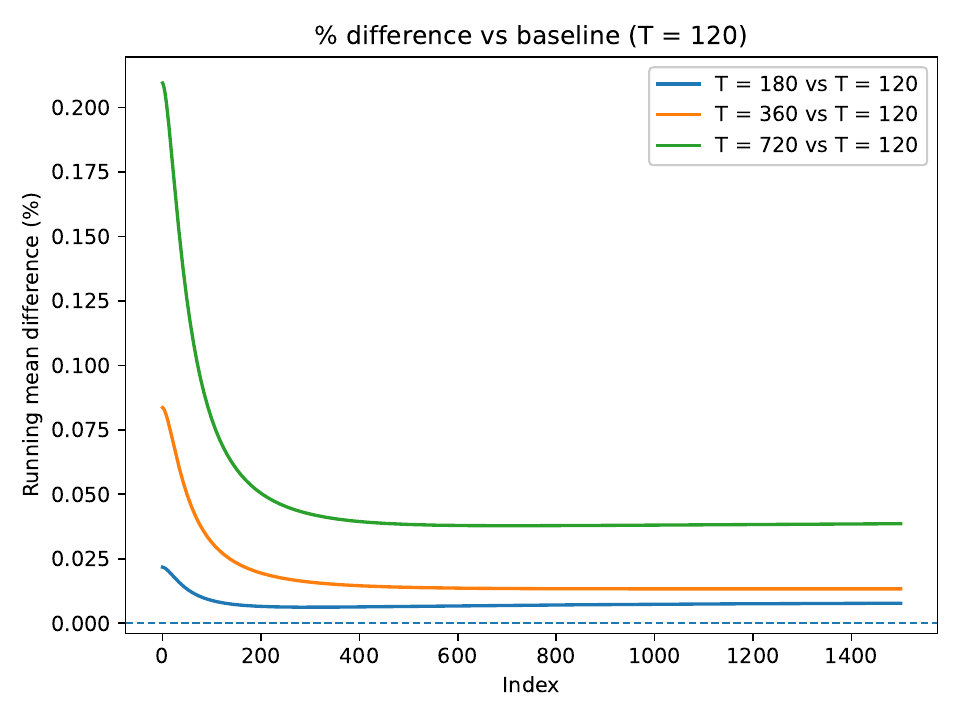}
{\small (a)}
\end{minipage}
\begin{minipage}[b]{0.48\textwidth}
\centering
\includegraphics[width=\linewidth]{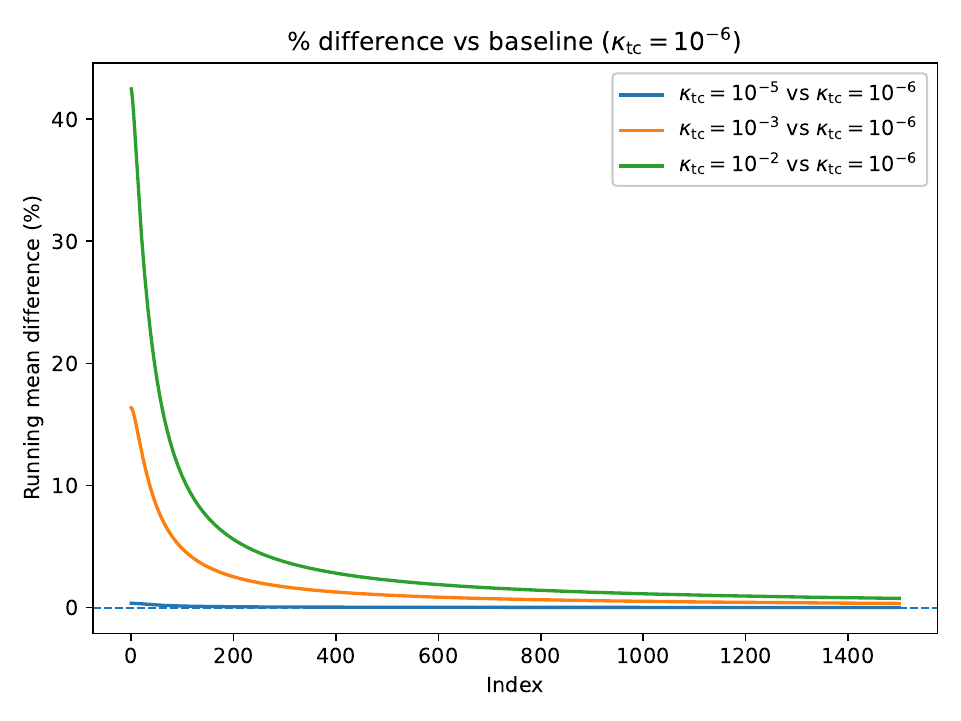}
{\small (b)}
\end{minipage}\hfill
\vspace{0.6em}
% -------- Row 2 --------
\begin{minipage}[b]{0.48\textwidth}
\centering
\includegraphics[width=\linewidth]{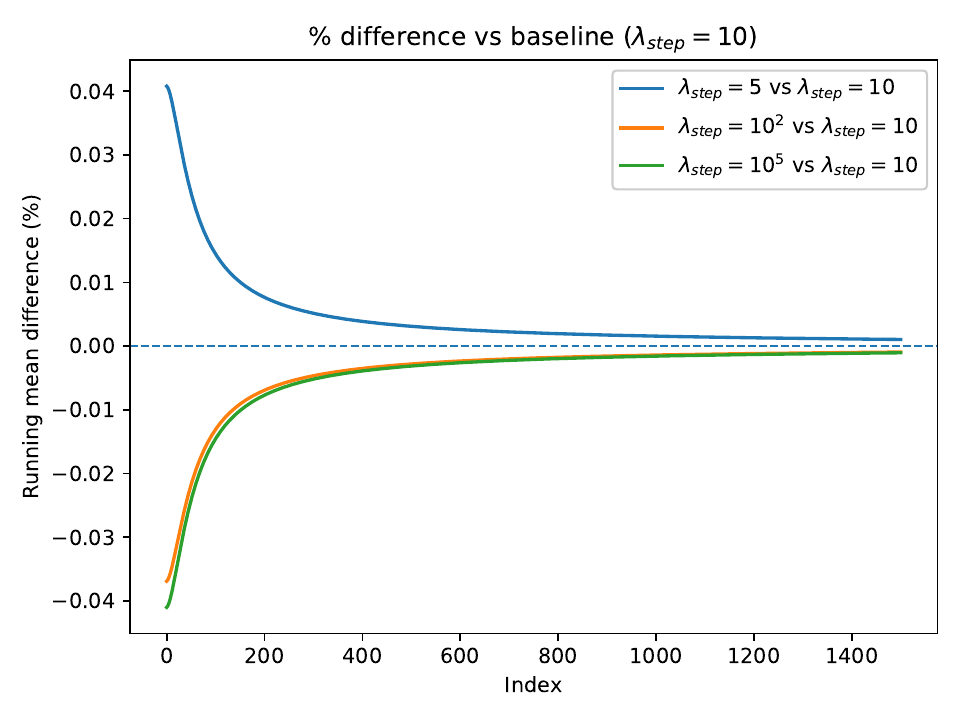}
{\small (c)}
\end{minipage}
\begin{minipage}[b]{0.48\textwidth}
\centering
\includegraphics[width=\linewidth]{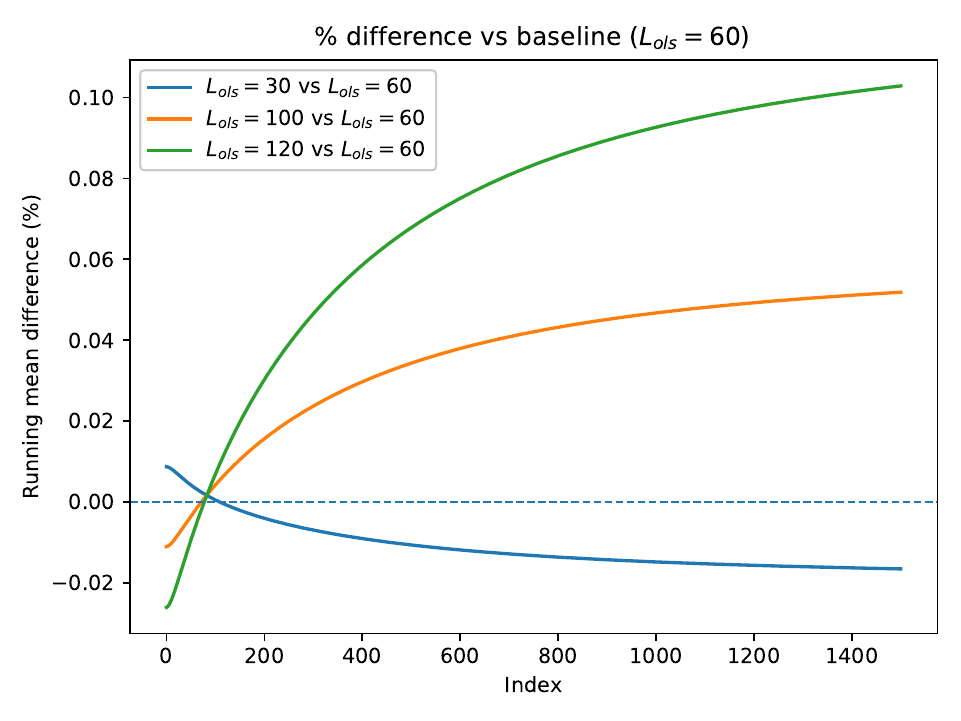}
{\small (d)}
\end{minipage}\hfill
\caption{DRO-LQ: policy gradient convergence and gain stabilization.}
\label{figure5}
\end{figure}

\noindent $\mathbf{T}$. In Figure \ref{figure5}(a), the running-mean gaps start large and positive and then converge toward a fixed level; moreover, larger $T$ produces a larger initial gap, consistent with a modest but systematic horizon effect rather than instability, and with attenuating marginal benefit as $T$ grows. \\

\noindent $\mathbf{\kappa_{t c}}$. When $\kappa_{t c}$ rises materially above baseline, the start gap can approach $40 \%$, then declines and is near zero in the tail; larger $\kappa_{t c}$ penalizes control effort more, shrinking gains and slowing trading, so curves separate early and flatten as the optimizer settles.\\

\noindent $\mathbf{\lambda_{\text{step }}}$. Smaller $\lambda$ yields a positive signed difference; much larger $\lambda$ yields a negative one; the start gap is clear but the tail drifts to 0, and pushing $\lambda$ from 10 to $10^3-10^5$ changes results only slightly. \\

\noindent $\mathbf{L_{ols}}$. The start is small, but the tail becomes larger; from Figure \ref{figure5}(d), we can find larger $L$ produces a larger positive tail. Since $L_{ols}$ sets the regression window for $a_t^*$ with $w_t=\Delta a_t$, longer windows smooth exposures and shrink empirical $\Sigma_w$; shorter windows are more reactive and noisier, so the smoothed estimator induces a stable offset that persists. 

\section{Conclusion}
This paper develops a perturbance-wise framework for finite horizon LQ control that unifies the classical model, a constraint embedded affine policy class, and a distributionally robust (Wasserstein) variant under a single learning view. By working in an augmented representation, we obtain ARE backward recursions with nonzero-mean noise and an exact policy gradient formula; treating empirical moments as fixed yields a clean descent direction. We further prove global convergence of constant stepsize gradient updates with step sizes chosen as simple polynomials of problem parameters.

Empirically, the mean–variance study clarifies when feasibility alters the unconstrained solution and when it is essentially neutral. On dynamic benchmark tracking with real data, the DRO controller trained by policy gradient exhibits stable, monotone cost decrease toward the baseline and remains well-behaved under realistic scaling of $Q$ and $R$. Sensitivity analyses show intuitive trade-offs: longer horizons improve tracking but increase turnover; higher trading cost intensity regularizes gains and slows adjustment; stronger state penalties reduce variance at the expense of control effort; longer estimation windows smooth but delay adaptation. Overall, the results provide an implementable toolkit for finite time LQ learning under constraints and distributional ambiguity, and point toward extensions to richer ambiguity sets, chance constraints, partial observation, and variance adaptive policy gradient schemes.

\end{document}